\documentclass[11pt,leqno]{amsart}

\usepackage[all,cmtip]{xy}
\usepackage{a4,latexsym}
\usepackage[utf8]{inputenc}
\usepackage{amsfonts}
\usepackage[hidelinks]{hyperref}
\usepackage{amsxtra,amsmath,amsfonts,amscd, amssymb, mathrsfs, amsthm}
\usepackage{color, mathtools}
\usepackage{bbm}
\usepackage{enumitem}
\usepackage{xcolor}
\usepackage{graphicx}
\usepackage{tikz-cd}
\usepackage{braket}% \Set command
\usepackage[all]{xypic}
\usepackage[toc,page,title,titletoc,header]{appendix}
\usepackage[capitalize]{cleveref}

\numberwithin{paragraph}{section}

\setlist[enumerate]{label=\it{(\roman*)},
	ref=\it{(\roman*)}}

\newcommand{\R}{{\mathbb R}}

\newtheorem{theorem}{Theorem}[section]

\newtheorem{lemma}[theorem]{Lemma}
\newtheorem{lemma*}{Lemma}

\theoremstyle{definition}
\newtheorem{definition}[theorem]{Definition}
\newtheorem{proposition&definition}[theorem]{Proposition\&Definition}
\newtheorem{lemma&definition}[theorem]{Lemma\&Definition}
\newtheorem{theorem&definition}[theorem]{Theorem\&Definition}

\newtheorem{example*}{Example}
\newtheorem{remark}{Remark}

\newtheorem{question*}{Question}

\def\vm{{\vec m}}
\def\vn{{\vec n}}
\def\p{{\partial}}

\def\I{\mathcal I}
\def\J{\mathcal J}
\def\K{\mathcal K}

\numberwithin{equation}{section}

\begin{document}
	
	\title[ A class  of  multi-parameter Fourier integral operators: endpoint Hardy space bounds]{ A class  of  multi-parameter Fourier integral operators: endpoint Hardy space bounds }
	
	\author[J.~Cheng]{Jinhua Cheng}
	\address{J. Cheng,  Department of Mathematics, Westlake University, 310024 Hangzhou, P. R. China }
	\email{chengjinhua@westlake.edu.cn}
	
	%\author[W.~Gubler]{Walter Gubler}
	%\address{W. Gubler, Mathematik, Universit{\"a}t 
		%Regensburg, 93040 Regensburg, Germany}
	%\email{walter.gubler@mathematik.uni-regensburg.de}

	\begin{abstract}
In this paper we study  a class of Fourier integral operators, whose symbols lie in the  multi-parameter H\"ormander class $S^{\vec m}(  \mathbb{R}^\vn)$, where ~$\vec m=(m_1,m_2,\dots,m_d)$ is the order. We show that if in addition the phase function $\Phi(x,\xi)$ can be written as $\Phi(x,\xi)=\sum_{i=1}^d\Phi_i(x_i,\xi_i)$, and each $\Phi_i(x_i,\xi_i)$ satisfies the non-degeneracy condition, then   such Fourier integral operators with  order  ~$\vec m=(-(n_1-1)/2, -(n_2-1)/2,\dots, -(n_d-1)/2)$  are actually bounded from rectangular  Hardy space $H_{rect}^1(\mathbb{R}^\vn)$ to $L^1( \mathbb{R}^n )$.

	\end{abstract}
	
	
	\maketitle
	
	\setcounter{tocdepth}{1}
	
	%\tableofcontents

%----------------------------------------------------------------------------------------
%	SECTION 1
%----------------------------------------------------------------------------------------	
\section{Introduction}
\setcounter{equation}{0}
Let $f\in \mathcal{S}(\mathbb{R}^n)$ be a Schwartz function, a Fourier integral operator defined by H\"ormander is in the form
\begin{equation}\label{linear FIO}
T  f(x)=\int_{\mathbb{R}^n}e^{2\pi i  \Phi(x,\xi)}\Hat{f}(\xi)\sigma(x,\xi)d\xi,
\end{equation}
where $\Hat{f}(\xi)=\int_{\mathbb{R}^n} e^{-2\pi i  x\cdot \xi} f(x)dx $  is the Fourier transform of $f$ and   the symbol 
$\sigma(x,\xi)\in \mathcal{C}^\infty(\mathbb{R}^n\times \mathbb{R}^n )$ has a compact support in the $x$-variable, belonging to the H\"ormander class $S^m(\mathbb{R}^n)$, that is,  
 satisfying the estimates
\begin{equation}
| \partial_\xi^\alpha \partial_x^\beta \sigma(x,\xi) |\le C_{\alpha,\beta}(1+|\xi|)^{m-|\alpha|}
\end{equation}
for multi-indices $\alpha,\beta$, and $m\in \mathbb{R}$ is  called the order of the symbol and the operator. The phase function $\Phi(x,\xi)\in \mathcal{C}^{\infty}( \mathbb{R}^n \times (\mathbb{R}^n\setminus\{0\})   )$ is a real-valued, positively homogeneous of degree one in the variable $\xi$, i.e.,
\[
\Phi(x,t\xi)=t\Phi(x,\xi),\quad t>0,\quad x\in \mathbb{R}^n.
\]
What's more, $\Phi$ satisfies the \emph{non-degenercy} condition

\begin{equation}
\det\Big(  \frac{\partial^2 \Phi(x,\xi)}{\partial x\partial \xi} \Big) \neq 0,~~~~~\xi\in \mathbb{R}^n\setminus\{0\}.
\end{equation}
The local regularity properties of Fourier integral operators have been studied extensively and intensively  over the past decades. It is well-known that   $T$  of order $0$ is bounded on $L^2$, as was shown by Eskin \cite{E70}, also  H\"ormander \cite{H71},     the first  systematic treatment of Fourier integral operators. On the other hand, for $p\neq 2$, Fourier integral operators may not  be bounded on $L^p$. Sharp $L^p$-estimate for Fourier integral operators was due to Seeger, Sogge and Stein \cite{SSS91}, where they showed the operator $ T $ of order $m$ is bounded on $L^p(\mathbb{R}^n)$ if $m\le -(n-1)|1/p-1/2|$~ for $~1<p<\infty$. Due to the characterization of the Hardy space $H^1(\mathbb{R}^n)$
and the complex interpolation theorem obtained  by Fefferman and Stein, see \cite{F71} and \cite{FS72},  they essentially proved that $ T $ of order $-(n-1)/2$ maps the Hardy space $H^1(\mathbb{R}^n)$ into $L^1(\mathbb{R}^n)$.

 The classical theory of harmonic analysis may be described as around the Hardy–Littlewood maximal operator and its relationship with certain singular integral operators which commute with the classical one-parameter family dilations $\delta: x\rightarrow \delta x=(\delta x_1, \dots, \delta x_d), ~\delta >0$, while  multi-parameter theory, sometimes called product theory corresponds to range of questions which are concerned with issues of harmonic analysis that are invariant with respect to a family of dilations $\delta: x\rightarrow \delta x=(\delta_1 x_1, \dots, \delta_d x_d), ~\delta_i >0, i=1,\dots,d$.  The multi-parameter function spaces and boundedness of Fourier multipliers, singular integral operators on such spaces have been extensively studied by many authors over the past decades. For instance, \cite{MRS95, MRS96, MPTT04, MPTT06, W22,W23}.

The purpose of the present paper is to prove a multi-parameter analogue of the Seeger-Sogge-Stein theorem. Let $\sigma(x,\xi)\in \mathcal{C}^\infty(\mathbb{R}^n\times \mathbb{R}^n)$ has a compact support in $x$-variable,  we say $\sigma$ belongs to multi-parameter H\"ormander class $S^{\vec m}(  \mathbb{R}^\vn ), \mathbb{R}^\vn:=\mathbb{R}^{n_1} \times \mathbb{R}^{n_2} \times\cdots\times \mathbb{R}^{n_d} ,~\vec m=(m_1,m_2,\dots, m_d)$ if it satisfies the differential inequalities

\begin{equation}\label{product symbol}
| \partial_\xi^\alpha \partial_x^\beta \sigma(x,\xi) |\le C_{\alpha,\beta} \prod_{i=1}^d (1+|\xi_i|)^{m_i-|\alpha_i|}
\end{equation}
for all multi-indices $\alpha,\beta$ and $\xi=(\xi_1,\xi_2,\dots,\xi_d)\in  \mathbb{R}^\vn,~n=\sum_{i=1}^d n_i $.
 A multi-parameter Fourier integral operator is of the same form of a standard Fourier integral operator except that the symbol $\sigma$ now belongs to the  multi-parameter H\"ormander class  $S^{\vec m}(  \mathbb{R}^\vn ) $.  
 
   Similar to one-parameter theory, in order to get sharp results of multi-parameter Fourier integral operators, we need to consider the product Hardy space $H^1(\mathbb{R}^\vn)$,  which was first introduced by Gundy and Stein \cite{GS79}, then developed by Chang and R. Fefferman,  see \cite{CF80, CF82, CF85}, also see  R. Fefferman \cite{F85}-\cite{Fe85}.  However, the atomic decomposition of  the product Hardy space is much more complicated.  It is conjectured at first that the product Hardy space can be represented by rectangle atoms, a function $a(x_1,\dots, x_d)$ supported on a rectangle $R=\prod_{i=1}^d I_i$, $I_i$ is a cube in $\mathbb{R}^{n_i}$, is   called an rectangle atom provided that $||a||_{L^2} \le |R|^{-1/2}$ and $\int_{\mathbb{R}^{n_i}} a(x)dx_i=0, ~1\le i\le d $. 
 Then out of expectation,  Carleson's  counterexample \cite{C74} tells us  that the rectangular Hardy space is a proper subspace of product Hardy space. Actually Chang and Fefferman \cite{CF80} showed  each product Hardy space atom is  supported in  an arbitrary open set of finite measure,  and this prevents us from proving $H^1\rightarrow L^1$ boundedness of operators by checking the actions on rectangle atoms.

Although the product Hardy space cannot be decomposed in the right way as a linear combination of rectangle atoms, nevertheless, for certain  $L^2$ bounded singular integral  operators, for example, Journ\'e's class, R. Fefferman \cite{F86, F87} showed it does suffice to check the action of $T$ on rectangle atoms by applying the remarkable covering lemma proved by Journ\'e, see \cite{J86}. As one-parameter theory, multi-parameter Fourier integral operators in general are not pseudo-local, hence they are not in Journ\'e class.   Recently,  Bajas,   Latorre and  Rinc\'on and Wright \cite{BLRW20} pointed out, for the multi-parameter oscillatory singular integral operators, to prove boundedness on product Hardy space probably requires a new, alternate approach. This motivates us to obtain bounds first on the rectangular Hardy space $H_{rect}^1(\mathbb{R}^\vn)$.

Now we restrict the phase function to be the particular form, that is,
\begin{equation}\label{phase function}
\Phi(x,\xi)=\sum_{i=1}^d \Phi_i(x_i,\xi_i),
\end{equation}
where each $\Phi_i(x_i,\xi_i)\in  \mathcal{C}^{\infty}( \mathbb{R}^n \times (\mathbb{R}^n\setminus\{0\})   )$ is a real-valued, positively homogeneous of degree one in the variable $\xi_i$, and satisfies the \emph{non-degenercy} condition for $i=1,2,,\dots,d.$
Our main result is stated as below

\begin{theorem}\label{main theorem}
Let the Fourier integral operator $T$ be defined as (\ref{linear FIO}), ~(\ref{product symbol})-(\ref{phase function}), 
~$ \vec m=(-(n_1-1)/2, -(n_2-1)/2,\dots, -(n_d-1)/2)$ ,  and  $n=n_1+n_2+\cdots+n_d$,  then 
\begin{equation}\label{aim} 
||T f||_{L^1(\mathbb{R}^n )}\le C  ||f||_{ H_{rect}^1( \mathbb{R}^\vn )},~~~\sigma\in S^{\vec m}(  \mathbb{R}^\vn ).
\end{equation}

\end{theorem}

\begin{remark}
Throughout the paper, the constant $\mathfrak{C}$ may be dependent on the $\sigma,\Phi, \alpha, \beta$, $d, p, n$, but not dependent on $f$ or other specific function and  may be different line by line.
\end{remark}

The paper is organized as follows.  In section $2$, we shall introduce the necessary notation and make a multiple Littlewood-Paley decomposition of $T$. In section $3$, we will further decompose the operator $T$ by Seeger-Sogge-Stein decomposition. Section $4$ will be devoted of the majorization of the kernels of $T$. In section $5$, we prove the $L^2$-boundedness of $T$ and in the last section shall discuss further $H^1_{rect} \rightarrow L^1$-boundedness of more general operator $T$.

%\subsection*{Notation and terminology}

%----------------------------------------------------------------------------------------
%	SECTION 2
%----------------------------------------------------------------------------------------
\section{Preliminaries}
\setcounter{equation}{0}
A direct generalization of the  Hardy space $H^1(\R^n)$ to multi-parameter setting is the rectangular Hardy space $H_{rect}^1(\R^\vn)$.
We define the rectangular Hardy space by its atomic representation.
\begin{definition}
A rectangle atom on $\R^\vn$ is a function $a_R(x)$ supported on a rectangle $R=\prod_{i=1}^d I_i$ having the property 
\[
||a_R||_{L^2(\R^n)} \le |R|^{-1/2},
\]
and for each $1\le i \le d$
\[
\int_{\R^{n_i}} a_R(x_1,x_2, \dots,x_d) dx_i =0.
\]
\end{definition}
\begin{definition}
The rectangular Hardy space $H_{rect}^1(\R^\vn)$ is the space of functions $\sum\lambda _k a_k$ with each $a_k$  a rectangle atom and $\sum_{k} |\lambda_k|<\infty$.
 \end{definition}

Though this space seems the most natural generalization to higher dimensions, the rectangular Hardy space is too small that we cannot find a similar complex interpolation theorem which is crucial to obtain the $L^p$-estimates for multi-parameter operators arising from harmonic analysis. However, due to Chang and Fefferman,   the product Hardy space  is a suitable generalization for $H^1(\R^n) $.

Then immediately from the atomic decomposition above,  (\ref{aim}) reduces to 
\begin{equation}\label{aim 2}
|| T a_R||_{L^1( \mathbb{R}^{n})}\le \mathfrak{C},~~~\sigma\in S^{\vec m}(  \mathbb{R}^\vn),
\end{equation}
where $\vec m=(-(n_1-1)/2, -(n_2-1)/2,\dots, -(n_d-1)/2)$ and $a_R$ is an arbitrary   rectangle  atom.
Now fix an atom $a_R (x)$ supported  on  a rectangle  $R \subset \mathbb{R}^n$, we can actually assume that the measure of $R$,  $|R|<1$, otherwise, the estimate is trivial.  Because of our assumption the symbol $\sigma(x,\xi)$ has a compact support in $x$-variable, we have
\begin{equation}\label{trivial estimate}
||T  a_R ||_{L^1}\le C  ||T  a_R ||_{L^2}\le  C  ||a_R||_{L^2}\le  C  |R|^{-1/2}\le  C.
\end{equation}
The first inequality holds because $T $ has fixed compact support and Cauchy-Schwarz inequality; the second follows from the well-know $L^2$-boundedness of $T $ of order $0$.

Now    Let 
$R=\prod_{i=1}^d I_i,~~~~I_i\subset \mathbb{R}^{n_i}$ is a cube with center $\overline{y}_i$ for $i=1,2,\dots,d.$ Let $r_i$ be the side length of the cube $I_i$,  from (\ref{trivial estimate}) we see that  whether $r_i<1$ is crucial, hence  for convenience,  we define index sets $\mathcal{I},~\mathcal{J}, ~\K$ such that
\begin{equation}\label{def I}
\begin{array}{lc}\displaystyle
\K =\{ i \in \{1,2,\dots,d\}: 0<r_i<1  \},~~~~ \I \cup \mathcal{J}=\{1,2,\dots,d\} ,\quad \I \cap \J =\emptyset.
\end{array}
\end{equation}
Moreover
\begin{equation}
\begin{array}{lc}\displaystyle
\I \subset \K, \qquad 
\mathcal{J}_1=\mathcal{J} \cap  \K, \qquad \mathcal{J}_2=\{ i \in \{1,2,\dots,d\}: r_i \ge 1  \}.
\end{array}
\end{equation}

Let   $\varphi$  be  a smooth bump function on $\mathbb{R}$ such that
\begin{equation}\label{varphi}
\varphi(t)=1 \quad \textit{for} \quad |t|\le 1; \qquad \varphi(t)=0, \quad \textit{for}  \quad |t|\ge 2
\end{equation}
and define
\begin{equation}\label{phi_j}
\phi_{j_i}(\xi_i) =\varphi(2^{-j_i}|\xi_i|) -\varphi(2^{-j_i+1}|\xi_i|), \quad j_i\in \mathbb{Z}, \quad j_i>0,
\end{equation}
and  $\phi_{0}(\xi_i) =\varphi(|\xi_i|)$
for $\xi_i\in \mathbb{R}^{n_i},~i=1,2,\dots,d$. Let   $j=(j_1, j_2,\dots, j_d)$ and  let $\mathcal{A} =\{i_1,i_2,\dots,i_{|\mathcal{A}|} \}$ be a subset of $\{1,2,\dots,d \}$, we denote 
\[
\begin{array}{cc} \displaystyle
\xi_\mathcal{A}=(\xi_{i_1},\xi_{i_2}, \dots, \xi_{i_{|\mathcal{A}|}} ),\quad
x_\mathcal{A}=(x_{i_1},x_{i_2}, \dots, x_{i_{|\mathcal{A}|}} ),\quad j_\mathcal{A}=(j_{i_1},j_{i_2}, \dots, j_{i_{|\mathcal{A}|}} ),\quad 
\\\\ \displaystyle
\quad k_\mathcal{A}=(k_{i_1}, k_{i_2}, \dots, k_{i_{|\mathcal{A}|}} ),\quad 
n_\mathcal{A}=\sum_{i\in \mathcal{A}} n_i,\quad \Phi_\mathcal{A}(x_\mathcal{A},\xi_\mathcal{A} )=\sum_{i\in\mathcal{A}} \Phi_i(x_i,\xi_i),
\\\\ \displaystyle
\sum_{j_\mathcal{A}\ge k_\mathcal{A} } =\prod_{i\in \mathcal{A} } \sum_{j_i\ge k_i},
\quad \sum_{j_\mathcal{A}\le k_\mathcal{A} } =\prod_{i\in \mathcal{A} } \sum_{j_i\le k_i},\quad \phi_{j_\mathcal{A}  }(\xi_\mathcal{A})=\prod_{i\in \mathcal{A}}\phi_{j_i}(\xi_i),\quad 
\p_{\xi_\mathcal{A}}^{\alpha_\mathcal{A}}=\prod_{i \in \mathcal{A}} \p_{\xi_i}^{\alpha_i}
\end{array}
\]
Therefore, we have a multiple partition of unity,
\begin{equation}
\sum_{j_\mathcal{I}\ge 0 }\phi_{j_\mathcal{I} }(\xi_\mathcal{I}) =
\prod_{i\in \mathcal{I}} \sum_{j_i\ge 0} \phi_{j_\mathcal{I} }(\xi_\mathcal{I})=1.
\end{equation}
 Note that  $\phi_{j_\mathcal{I}}(\xi_{\mathcal{I}} )$ supported in the product dyadic shells,
\begin{equation}\label{support}
\{\xi_{\mathcal{I}} \in  \mathbb{R}^{n_{\mathcal{I}}} : 2^{j_i-1} \le |\xi_i|\le 2^{j_i+1},~~i\in \mathcal{I}\},
\end{equation}
and it satisfies the following differential inequalities
\begin{equation}
\Big|   \partial_{\xi_\mathcal{I}}^{\alpha_\I}  \phi_{j_\mathcal{I}} (\xi_\mathcal{I})\Big|\le 
C_{\alpha} \prod_{i\in \mathcal{I}} \Big( \frac{1}{|\xi_i|} \Big)^{|\alpha_i|}
\end{equation}
for all multi-indices $\alpha_i,~ i\in \mathcal{I}$.
We now define partial operators
\begin{equation}\label{rectangular partial operator}
T_{j_\mathcal{I}}f(x)=\int_{\mathbb{R}^n} f(y)K_{j_\mathcal{I}}(x,y)dy,
\end{equation}
where 
\begin{equation}
K_{j_\mathcal{I}  }(x, y)=\int_{\mathbb{R}^n} e^{2\pi  i ( \Phi(x,\xi)-y\cdot \xi ) }\sigma(x,\xi) \phi_{ j_\mathcal{I} }(\xi_\I) d\xi.
\end{equation}

%%~~~~~~~~~~~~~~~~~~~~~~~~~~~~~~Multi-parameter Seeger-Sogge-Stein decomposition~~~~~~~~~~~~~~~~~~~~~~~~~~~~ %%
\section{{Seeger-Sogge-Stein decomposition}}\label{multi-parameter Seeger-Sogge-Stein decomposition}
 For each $i\in \{1,2,\dots, d\}$, let 
$j_i>0$ be a positive integer, then  we  consider a collection of points $\{ \xi_{j_i}^{\nu_i} \}_{\nu_i} \in \mathbb{R}^{n_i}$, uniformly distributed on the unit sphere $\mathbb{S}^{n_i-1}$ with grid length equal to $2^{-j_i/2}$ multiplied by some suitable constant. Then 
for every given $\xi_i\in \mathbb{R}^{n_i}$, there exists a $\xi^{\nu_i}_{j_i}$ such that 
\[
\Big| \frac{\xi_i}{|\xi_i|}  -\xi^{\nu_i}_{j_i}\Big|\le 2^{-\frac{j_i }{2}}.
\]
On the other hand, there are at most a constant multiple of $2^{j_i\frac{n_i-1}{2}}$ elements in the collection $\{ \xi_{j_i}^{\nu_i} \}_{\nu_i} $.
Recall the definition of $\varphi$ in (\ref{varphi}), define
\begin{equation}
\varphi_{j_i}^{\nu_i} (\xi_i) =\varphi\Big(  2^{j_i/2}  \Big| \frac{\xi_i}{|\xi_i|}-\xi^{\nu_i}_{j_i}\Big|  \Big),
\end{equation}
whose support lies inside 
\[
\Gamma_{j_i}^{\nu_i}=\Big\{ \xi_i\in \mathbb{R}^{n_i}: \Big| \frac{\xi_i}{|\xi_i|}-\xi^{\nu_i}_{j_i}\Big|  \le 2\cdot 2^{-j_i/2}  \Big\}
\]
We also define
\begin{equation}\label{multi chi definition} 
\chi_{j_\I}^\nu(\xi_\I )= \prod_{i\in \I}\chi_{j_i}^{\nu_i}(\xi_i),~~~\chi_{j_i}^{\nu_i}(\xi_i)=\frac{\varphi_{j_i}^{\nu_i}(\xi_i)}{\sum_{\nu_i} \varphi_{j_i}^{\nu_i} (\xi_i)}.
\end{equation}
By a rotation,  we  choose axes in the $\xi_i$-space so that $\xi''_{i}$  is  in the direction of $\xi_{j_i}^{\nu_i} $ and $\xi'_i $ is perpendicular to $\xi_{j_i}^{\nu_i} $.
Then we have
\begin{equation}\label{multi chi estimate 1}
|\partial_{\xi''_{i}}^N \chi_{j_i}^{\nu_i}(\xi_i)|\le C_N|\xi_i|^{-N},~~~~~N\ge 1
\end{equation}
and
\begin{equation}\label{multi chi estimate 2}
|\partial_{\xi'_{i}}^\alpha \chi_{j_i}^{\nu_i} (\xi_i)|\le  C_\alpha 2^{|\alpha| j_i/2} |\xi_i|^{-|\alpha|}
\end{equation}
for multi-index $\alpha$. 

Take the Seeger-Sogge-Stein decomposition into consideration, we further define the partial operators
\begin{equation}\label{partial operator final}
T^\nu_{j_\mathcal{I}}f(x)=\int_{\mathbb{R}^n} f(y)K^\nu_{j_\mathcal{I}}(x,y)dy,
\end{equation}
where 
\begin{equation}
K^\nu_{j_\mathcal{I}  }(x, y)=\int_{\mathbb{R}^n} e^{2\pi  i ( \Phi(x,\xi)-y\cdot \xi ) }\sigma(x,\xi) \phi_{ j_\mathcal{I} }(\xi_\I) \chi^\nu_{j_\I}(\xi_\I) d\xi.
\end{equation} 
Thus immediately we see $ T_{j_\I}= \sum_{\nu} T^\nu_{j_\I} $ and $K_{j_\I} =\sum_{\nu} K^\nu_{j_\I}$.

Now we are in the position to define the \emph{region of influence} for multi-parameter setting.
We first define the rectangle $R_{j_i}^{\nu_i}$   be the set of all $x_i$ such that
\[
\Big| \Big<\overline{y}_i-\nabla_{\xi_i} \Phi_i(x_i, \xi_{j_i}^{\nu_i}  ), e_i''\Big> \Big|\le C 2^{-j_i}
\] 
for unit vector  $e_i''$  which  is the  same direction of $\xi_{j_i}^{\nu_i} $ and 
\[
\Big| \Big<\overline{y}_i-\nabla_{\xi_i} \Phi(x_i,  \xi_{j_i}^{\nu_i}  ) , e_i'\Big> \Big|\le C 2^{-j_i/2}
\] 
for  unit vectors $e_i'$  which are normal  to  $\xi_{j_i}^{\nu_i} $. Thus we see 
$R_{j_i }^{\nu_i}$ is now basically a rectangle having $1$ side of length $\mathfrak{C} 2^{-j_i}$ and $ n_i-1$ of length $\mathfrak{C} 2^{-j_i/2}$. 
Next  the mapping 
\[
x_i \mapsto y_i =\nabla_{\xi_i} \Phi_i(x_i, \xi_{j_i}^{\nu_i} )
\]
has a non-vanishing  Jacobian for each $\xi_{j_i}^{\nu_i} $ since $\Phi_i$ satisfies the non-degeneracy  condition. 
 Then we see 
\[
|R_{j_i}^{\nu_i}|\le  C_{\Phi_i}  2^{-j_i} 2^{-j_i(n_i-1)/2} .
\]
Let $k_i\ge 0$ be an integer such that 
\[
2^{-k_i} \le r_i \le 2^{-k_i +1}, \qquad i \in \K.
\]
The {\bf region of influence  }  in subspace $\R^{n_i}$ for multi-parameter Fourier integral operators is defined as
\begin{equation}
 Q_i=\bigcup_{j_i \ge k_i } \bigcup_{\nu }R_{j_i}^{\nu_i}, \quad i\in \K.
\end{equation}
Therefore, we have  $Q_i\subset \mathbb{R}^{n_i}$ and an elementary calculation shows 
\begin{equation}
|Q_i|\le  C  \sum_{j_i \ge k_i } |R_{j_i}^{\nu_i} |2^{j_i(n_i-1)/2} \le  C_{\Phi_i}  r_i.
\end{equation}
Now  define 
\[
Q_i^c=\mathbb{R}^{n_i} \setminus Q_i, ~~~~Q_{\mathcal{I}}^c= \bigotimes_{i\in\mathcal{I}} Q_i^c, ~~~~Q_{\mathcal{J}_1}= \bigotimes_{i\in\mathcal{J}_1} Q_i.
\]
Therefore, we have another "partition of unity",
\begin{equation}
\begin{array}{lc}\displaystyle  
\int_{\mathbb{R}^n} |T a_R (x)| dx=\sum_{\mathcal{I} \subset \K } \int_{Q^c_{\mathcal{I}}} \int_{Q_{\mathcal{J}_1}} \int_{\mathbb{R}^{n_{\mathcal{J}_2}}} |T a_R  (x)| dx
\\\\ \displaystyle
=\sum_{\mathcal{I} \subset \K  } \sum_{j_\mathcal{I} \ge 0} \int_{Q^c_{\mathcal{I}}} \int_{Q_{\mathcal{J}_1}} \int_{\mathbb{R}^{n_{\mathcal{J}_2}}} |T_{j_\mathcal{I}} a_R(x)| dx,
\end{array}
\end{equation}
where the sum $\sum_{\mathcal{I}\subset \K }$ is taken over all subsets of $\K $.
Since there are only finitely many $\mathcal{I}$'s, it suffices to show that, for a fixed $\mathcal{I}$, we have
\begin{equation}
\sum_{j_\mathcal{I} \ge 0} \int_{Q^c_{\mathcal{I}}} \int_{Q_{\mathcal{J}_1}} \int_{\mathbb{R}^{n_{\mathcal{J}_2}}} | T_{j_\mathcal{I}} a_R (x)| dx \le C.
\end{equation}
Applying Cauchy-Schwarz inequality and note that $\sigma(x,\xi)$ has compact support in $x$-variable, we have 
\begin{equation}\label{rectangular aim}
\sum_{j_\mathcal{I} \ge 0} \int_{Q^c_{\mathcal{I}}} \int_{Q_{\mathcal{J}_1}} \int_{\mathbb{R}^{n_{\mathcal{J}_2}}} | T _{j_\mathcal{I}} a_R(x)| dx 
\le |Q_{\mathcal{J}_1}|^{\frac{1}{2}} \sum_{j_\mathcal{I} \ge 0}   \int_{Q^c_{\mathcal{I}}} 
\Big \{\int_{\mathbb{R}^{n_{\mathcal{J}}}} |T_{j_\mathcal{I}} a_R(x)|^2 dx_\mathcal{J} \Big\}^{\frac{1}{2}} dx_\mathcal{I}.
\end{equation}

 \section{Majorization of the kernels }\label{majorization on hardy space}
Remember in one-parameter theory \cite{SSS91},  when the radius of the support of the atom is large, then it is trivial to apply the $L^2$-boundedness of $T$, while then the radius is small, then in order to obtain $L^1$-estimate,  we need majorize the kernels of $T$.  
However, for a fix rectangle atom $R$, though the measure $|R|<1$, there may be  some $r_i\le 1$ while some $r_i >1$. Therefore, for the multi-parameter setting,  we need a mixed $L^1$ and $L^2$ estimate of $T$.

We aim to show the following key  estimates of the kernels.
\begin{lemma}\label{rectangular majorization}
Suppose that $\sigma(x,\xi) \in S^{\vec m}( \R^\vn)$,  where $~\vec m=(-(n_1-1)/2, -(n_2-1)/2,\dots, -(n_d-1)/2)$. Let $\mathcal{I}_1 $ be a subset of $\mathcal{I}$ and $\mathcal{I}_2= \mathcal{I} \setminus \mathcal{I}_1$,
we have
\begin{equation}
\begin{array}{lc}\displaystyle
\int_{ Q^c_\mathcal{I} } 

\Big\{
\int_{ \mathbb{R}^{n_\mathcal{J}} } 
\Big|  T_{j_\mathcal{I}} a_R(x) \Big|^2 d x_\mathcal{J}
\Big\}^{\frac{1}{2}}
d x_\mathcal{I}
\le  C 
\prod_{i\in \mathcal{I}_1}2^{-j_i+k_i}
\prod_{i\in \mathcal{I}_2}2^{j_i-k_i}
\\\\ \displaystyle
  \prod_{i\in \mathcal{J}_1} |I_i|^{\frac{1}{2} -\frac{1}{2n_i}}
 \int_{ \mathbb{R}^{n_\mathcal{I}} } 
 \Big\{
 \int_{ \mathbb{R}^{n_\mathcal{J}} } 
\Big|a_R (x_\mathcal{I}, x_\mathcal{J} ) 
\Big|^2
dx_{\mathcal{J}} \Big\}^{\frac{1}{2}}
dx_\mathcal{I}.
\end{array}
\end{equation}

\end{lemma}

\begin{proof}
Recall the definition of the partial operator $T_{j_\I}^\nu $ in (\ref{partial operator final}) and rewritte
\begin{equation}
\begin{array}{lc} \displaystyle
T^\nu_{j_\mathcal{I}} a_R(x)=\int_{ \mathbb{R}^{n_\mathcal{I}} } 
\int_{ \mathbb{R}^{n_\mathcal{I}} } 
e^{2\pi  i  ( \Phi_\mathcal{I}(x_\mathcal{I}, \xi_\mathcal{I}) - y_\mathcal{I} \cdot \xi_\mathcal{I})}
G^\nu_{j_\mathcal{I}} a_R (x, y_\mathcal{I}, \xi_\mathcal{I} ) 
 d\xi_\mathcal{I} dy_\mathcal{I} ,
  \end{array}
\end{equation}
where 
\[
\begin{array}{lc}\displaystyle
G^\nu_{j_\mathcal{I}} a_R (x, y_\mathcal{I}, \xi_\mathcal{I} ) =
 \int_{ \mathbb{R}^{n_\mathcal{J}} }
 e^{2\pi i  \Phi_\mathcal{J}(x_\mathcal{J}, \xi_\mathcal{J}) }
 \widehat{a_R}(y_\mathcal{I}, \xi_\mathcal{J}  )
 \sigma(x,\xi) \phi_{j_\mathcal{I}}(\xi_\mathcal{I}) \chi^\nu_{j_{\mathcal{I}}}(\xi_\mathcal{I}) d\xi_\mathcal{J} ,
 \\\\ \displaystyle
  \widehat{a_R}(y_\mathcal{I}, \xi_\mathcal{J}  )= \int_{ \mathbb{R}^{n_\mathcal{J}} } e^{-2\pi  i y_\mathcal{J} \cdot \xi_\mathcal{J} } a_R(y_\mathcal{I}, y_\mathcal{J} )
dy_\mathcal{J}.
\end{array}
\]
Now for each $i\in \mathcal{I}$, we write similarly 
\begin{equation}
\Phi_i(x_i,\xi_i)-y_i\cdot\xi_i= [\nabla_{\xi_i}\Phi_i(x_i, \xi_{j_i}^{\nu_i} )-y_i ]\cdot \xi_i +\Psi_i(x_i,\xi_i),
\end{equation}
where
\begin{equation}
\Psi_i(x_i,\xi_i)=\Phi_i(x_i,\xi_i) -\nabla_{\xi_i}\Phi_i(x_i, \xi_{j_i}^{\nu_i}  )\cdot \xi_i,
\end{equation}
and from $\textbf 4.5$ chapter IX in  \cite{S93},  we see 
\begin{equation}\label{Psi estimate}
|\partial_{\xi''_{i}}^N \Psi_i(x_i,\xi_i)|\le C_N\cdot 2^{-Nj_i},~~|\partial_{\xi'_{i}}^\alpha \Psi_i(x_i,\xi_i)|\le C_N\cdot 2^{-|\alpha|j_i/2}
\end{equation}
for all multi-indices $\alpha$.
We now rewrite the kernel as
\begin{equation}
\begin{array}{lc}\displaystyle
 \Tilde{ G }^\nu_{j_\mathcal{I}} a_R (x, y_\mathcal{I}, \xi_\mathcal{I} ) 
 =
 e^{2\pi  i  \Psi_\I(x_\I,\xi_\I) } 
 G^\nu_{j_\mathcal{I}} a_R (x, y_\mathcal{I}, \xi_\mathcal{I} ) 
 \\\\ \displaystyle
 =
 \int_{ \mathbb{R}^{n_\mathcal{J}} }
 e^{2\pi  i   \Phi_\mathcal{J}(x_\mathcal{J}, \xi_\mathcal{J}) }
 \widehat{a_R}(y_\mathcal{I}, \xi_\mathcal{J}  ) \Theta^\nu_{j_\mathcal{I} }(x,\xi)
  d\xi_\mathcal{J} .
  \end{array}
\end{equation}
where
\begin{equation}\label{Theta}
  \Psi_\I(x_\I,\xi_\I) =\sum_{i\in \mathcal{I}}  \Psi_i(x_i,\xi_i),  \quad 
\Theta^\nu_{j_\mathcal{I} }(x,\xi)= e^{2\pi  i    \Psi_\I(x_\I,\xi_\I) } \sigma(x,\xi) 
\phi_{j_\mathcal{I} }(\xi_{\mathcal{I}})    \chi^\nu_{j_\mathcal{I}}(\xi_{\mathcal{I}}).
\end{equation}
We next introduce the operator $L$ defined by
\[
L=L_{\xi_{\mathcal{I}}}=\prod_{i\in \mathcal{I}} 
(I-2^{j_i/2 }\nabla_{\xi'_i} \cdot 2^{j_i/2 }\nabla_{\xi'_i}  ) (I-2^{j_i }\nabla_{\xi''_i} \cdot 2^{j_i }\nabla_{\xi''_i}  ) 
\]
Because of (\ref{multi chi definition})-(\ref{multi chi estimate 2}) and (\ref{Psi estimate})-(\ref{Theta}), and the fact that $\sigma(x,\xi) \in S^{\vec m}(  \mathbb{R}^\vn )$,  where $~\vec m=(-(n_1-1)/2, -(n_2-1)/2,\dots, -(n_d-1)/2)$,  we get that
\[
|L^N\Theta_{j_\mathcal{I}}^\nu (x,\xi)|\le C_N\prod_{i\in \mathcal{I}} 2^{-j_i(n_i-1)/2},~~~N\ge 0.
\]
On the other hand, the support of $\Theta_{j_\mathcal{I}}^\nu (x,\xi) $  in the variable $\xi_{\mathcal{I}}$ is 
\[
|\mathbf{supp}_{\mathcal{I}}\Theta_{j_\mathcal{I}}^\nu (x,\xi)|\le  C  \prod_{i\in \mathcal{I}} 2^{j_i} 2^{j_i(n_i-1)/2}.
\]
Thus integration by parts gives 
\begin{equation}
\begin{array}{lc} \displaystyle
T^\nu_{j_\mathcal{I}} a_R(x)=\int_{ \mathbb{R}^{n_\mathcal{I}} } 
\int_{ \mathbb{R}^{n_\mathcal{I}} } \prod_{i\in \mathcal{I}}
e^{2\pi i  [\nabla_{\xi_i}\Phi_i(x_i, \xi_{j_i}^{\nu_i}  )-y_i ]\cdot \xi_i }
\Tilde{G}^\nu_{j_\mathcal{I}} a_R (x, y_\mathcal{I}, \xi_\mathcal{I} ) 
 d\xi_\mathcal{I} dy_\mathcal{I} 
 \\\\ \displaystyle
= C  \int_{ \mathbb{R}^{n_\mathcal{I}} } 
\int_{ \mathbb{R}^{n_\mathcal{I}} } 
 \prod_{i\in \mathcal{I}} \Big\{1+2^{j_i} (\nabla_{\xi_i}\Phi_i(x_i,\xi^{\nu_i}_{j_i})-y_i )'' \Big\}^{-2N}
 \Big\{1+2^{j_i/2} (\nabla_{\xi_i}\Phi_i(x_i,\xi^{\nu_i}_{j_i})-y_i )' \Big\}^{-2N}
 \\\\ \displaystyle ~~~~~~~~~~~
\times  e^{2\pi i  [\nabla_{\xi_i}\Phi_i(x_i,\xi_{j_i}^{\nu_i})-y_i ]\cdot \xi_i }
L^N\Tilde{G}^\nu_{j_\mathcal{I}} a_R (x, y_\mathcal{I}, \xi_\mathcal{I} ) 
 d\xi_\mathcal{I} dy_\mathcal{I} .
  \end{array}
\end{equation}
Now use Minkowski integral inequality  and Lemma \ref{L^2 estimate} proved in the next section,
we have
\begin{equation}\label{Omega size}
\begin{array}{lc}\displaystyle
\Big\{
\int_{ \mathbb{R}^{n_\mathcal{J}} } \Big| T^\nu_{j_\mathcal{I}} a_R(x) \Big|^2 d x_\mathcal{J}
\Big\}^{\frac{1}{2}}
\\\\ \displaystyle
\le  C_N  \int_{ \mathbb{R}^{n_\mathcal{I}} } 
\int_{ \mathbb{R}^{n_\mathcal{I}} } 
 \prod_{i\in \mathcal{I}} \Big\{1+2^{j_i}|(\nabla_{\xi_i}\Phi_i(x_i,\xi^{\nu_i}_{j_i})-y_i )''|  \Big\}^{-2N}
 \Big\{1+2^{j_i/2}|(\nabla_{\xi_i}\Phi_i(x_i,\xi^{\nu_i}_{j_i})-y_i )'| \Big\}^{-2N}
 \\\\ \displaystyle ~~~~~~~~~~~
 \Big \{
 \int_{ \mathbb{R}^{n_\mathcal{J}} } 
\Big|
L^N\Tilde{G}^\nu_{j_\mathcal{I}} a_R (x, y_\mathcal{I}, \xi_\mathcal{I} ) 
\Big|^2
dx_{\mathcal{J}} \Big\}^{\frac{1}{2}}
 d\xi_\mathcal{I} dy_\mathcal{I} 
 \\\\ \displaystyle
\le  C_N  
\int_{ \mathbb{R}^{n_\mathcal{I}} } 
 \prod_{i\in \mathcal{I}} \Big\{1+2^{j_i}|(\nabla_{\xi_i}\Phi_i(x_i,\xi^{\nu_i}_{j_i})-y_i )''|  \Big\}^{-2N}
 \Big\{1+2^{j_i/2}|(\nabla_{\xi_i}\Phi_i(x_i,\xi^{\nu_i}_{j_i})-y_i )'| \Big\}^{-2N}
 \\\\ \displaystyle ~~~~~~~~~~~
 \times 
 \prod_{i\in \mathcal{I}} 2^{j_i} 2^{j_i\frac{n_i-1}{2}} \cdot 2^{-j_i\frac{n_i-1}{2}}
 \prod_{i\in \mathcal{J}_1} |I_i|^{\frac{1}{2} -\frac{1}{2n_i}}
 \Big\{
 \int_{ \mathbb{R}^{n_\mathcal{J}} } 
\Big|a_R (y_\mathcal{I}, x_\mathcal{J} ) 
\Big|^2
dx_{\mathcal{J}} \Big\}^{\frac{1}{2}}
dy_\mathcal{I} .

\end{array}
\end{equation}
Recall that $k_i$ is an integer such that 
\[
2^{-k_i} \le r_i\le 2^{-k_i+1},  \quad i\in \mathcal{I}.
\]
For each $i\in \mathcal{I}$, there is a unit vector $\xi_{k_i}^{\mu_i}$, so that $|\xi^{\nu_i}_{j_i} -\xi^{\mu_i}_{k_i}|\le 2^{-k_i/2}$. Since $Q_i=\bigcup_{j_i\ge k_i}\bigcup_{\nu_i }R_{j_i}^{\nu_i}$, we have the inclusion
\[
\mathbb{R}^{n_i}\setminus Q_i \subset \mathbb{R}^{n_i}\setminus R_{{k_i}}^{\mu_i}.
\]
Then from the definition of $R_{k_i}^{\mu_i}$, for each   $i \in\mathcal{I}$,  we have 
\[
 2^{k_i}\Big| \Big<\overline{y}_i-\nabla_{\xi_i} \Phi_i(x_i, \xi_{k_i}^{\mu_i}  ), e_i''\Big> \Big| >  C 
\] 
or 
\[
2^{k_i/2}  \Big| \Big<\overline{y}_i-\nabla_{\xi_i} \Phi(x_i, \xi_{k_i}^{\mu_i}  ) , e_i'\Big> \Big|>  C.
\] 
whenever $y_i \in B_{r_i}(\overline{y_i})$, then $|y_i - \overline{y_i}|\le 2^{-k_i+1}$, since we can choose  $C$  sufficiently large, we get as a consequence
\[
2^{j_i} \Big| \Big< y_i-\nabla_{\xi_i} \Phi_i(x_i, \xi_{j_i}^{\nu_i} ), e_i''\Big> \Big|\ge C 2^{j_i-k_i}
\] 
or 
\[
2^{j_i/2} \Big| \Big<y_i-\nabla_{\xi_i} \Phi_i(x_i, \xi_{j_i}^{\nu_i}  ), e_i'\Big> \Big|\ge C 2^{(j_i-k_i)/2}, 
\] 
when $ j_i\ge k_i$ for each $i\in \mathcal{I}$ and  $ x_i \in Q_i^c$. Inserting this in the bound (\ref{Omega size}) for $i\in \mathcal{I}_1$
 and arguing as before, we obtain
\begin{equation}\label{Omega size 2}
\begin{array}{lc}\displaystyle
\Big\{
\int_{ \mathbb{R}^{n_\mathcal{J}} } \Big| T^\nu_{j_\mathcal{I}} a_R(x) \Big|^2 d x_\mathcal{J}
\Big\}^{\frac{1}{2}}
 \\\\ \displaystyle
\le  C_N  
\int_{ \mathbb{R}^{n_\mathcal{I}} } 
 \prod_{i\in \mathcal{I}} \Big\{1+2^{j_i}|(\nabla_{\xi_i}\Phi_i(x_i,\xi^{\nu_i}_{j_i})-y_i )''|  \Big\}^{1-2N}
 \Big\{1+2^{j_i/2}|(\nabla_{\xi_i}\Phi_i(x_i,\xi^{\nu_i}_{j_i})-y_i )'| \Big\}^{2-2N}
 \\\\ \displaystyle ~\prod_{i\in \mathcal{I}_1} 2^{-j_i+k_i}
 \prod_{i\in \mathcal{I}} 2^{j_i} 2^{j_i\frac{n_i-1}{2}} \cdot 2^{-j_i\frac{n_i-1}{2}}
 \prod_{i\in \mathcal{J}_1} |I_i|^{\frac{1}{2} -\frac{1}{2n_i}}
 \Big\{
 \int_{ \mathbb{R}^{n_\mathcal{J}} } 
\Big|a_R (y_\mathcal{I}, x_\mathcal{J} ) 
\Big|^2
dx_{\mathcal{J}} \Big\}^{\frac{1}{2}}
dy_\mathcal{I} .
\end{array}
\end{equation}
In carrying out the estimate for  $\int_{ Q_{\mathcal{I}}^c} 
\Big\{
\int_{ \mathbb{R}^{n_\mathcal{J}} } 
\Big| T^\nu_{j_\mathcal{I}} a_R(x) \Big|^2 d x_\mathcal{J}
\Big\}^{\frac{1}{2}}
d x_\mathcal{I}
$, we use the majorization (\ref{Omega size 2}), as well as the change of variables
\[
x_i \mapsto \nabla_{\xi_i}\Phi_i(x_i,\xi_{j_i}^{\nu_i}),
\]
whose Jacobian is bounded from below, since each $\Phi_i(x_i,\xi_i)$ satisfies the non-degeneracy condition by our assumption. The result is
\begin{equation}
\begin{array}{lc}\displaystyle
\int_{ Q_{\mathcal{I}}^c} 
\Big\{
\int_{ \mathbb{R}^{n_\mathcal{J}} } 
\Big| T^\nu_{j_\mathcal{I}} a_R(x) \Big|^2 d x_\mathcal{J}
\Big\}^{\frac{1}{2}}
d x_\mathcal{I}
 \\\\ \displaystyle
\le  \mathfrak{C}_N   \prod_{i\in \mathcal{I}} 2^{-j_i\frac{n_i-1}{2}} 
\int_{ \mathbb{R}^{n_\mathcal{I}} } 
\Big\{
\int_{ Q^c_\mathcal{I}} 
 \prod_{i\in \mathcal{I}} \Big\{1+|(x_i-y_i )''|  \Big\}^{1-2N}
 \Big\{1+|(x_i-y_i )'| \Big\}^{2-2N} 
 dx_\mathcal{I} 
 \Big\}
 \\\\ \displaystyle~
 \times   ~\prod_{i\in \mathcal{I}_1} 2^{-j_i+k_i}
  \prod_{i\in \mathcal{J}_1} |I_i|^{\frac{1}{2} -\frac{1}{2n_i}}
 \Big\{
 \int_{ \mathbb{R}^{n_\mathcal{J}} } 
\Big|a_R (y_\mathcal{I}, x_\mathcal{J} ) 
\Big|^2
dx_{\mathcal{J}} \Big\}^{\frac{1}{2}}
dy_\mathcal{I} ,
\end{array}
\end{equation}
if we choose $N$ so that $2N>n+2$. Therefore
\begin{equation}
\begin{array}{lc}\displaystyle
\int_{ Q^c_\mathcal{I} } 
\Big\{
\int_{ \mathbb{R}^{n_\mathcal{J}} } 
\Big| T^\nu_{j_\mathcal{I}} a_R(x) \Big|^2 d x_\mathcal{J}
\Big\}^{\frac{1}{2}}
d x_\mathcal{I}
\\\\ \displaystyle
\le  \mathfrak{C}_N  
 \prod_{i\in \mathcal{I}} 2^{-j_i\frac{n_i-1}{2}} ~\prod_{i\in \mathcal{I}_1} 2^{-j_i+k_i}
  \prod_{i\in \mathcal{J}_1} |I_i|^{\frac{1}{2} -\frac{1}{2n_i}}
 \int_{ \mathbb{R}^{n_\mathcal{J}} } 
 \Big\{
 \int_{ \mathbb{R}^{n_\mathcal{J}} } 
\Big|a_R (x_\mathcal{I}, x_\mathcal{J} ) 
\Big|^2
dx_{\mathcal{J}} \Big\}^{\frac{1}{2}}
dx_\mathcal{I}.
\end{array}
\end{equation}
Since $\int_{\R^{n_i}}a_R(y) dy_i=0 $ for each $i=1,2,\dots,d$,  we have 
\[
\int_{\mathbb{R}^{n_i}} G^\nu_{j_\mathcal{I}}a_R(x,y_\mathcal{I}, \xi_\mathcal{I})dy_i=0,  \quad \textit{for} \quad i \in \mathcal{I}.
\]
We can rewrite  $ T^\nu_{j_\mathcal{I}} a_R(x)$ as 
\begin{equation}\label{Omega deformation}
\begin{array}{lc} \displaystyle
T^\nu_{j_\mathcal{I}} a_R(x)=\int_{ \mathbb{R}^{n_\mathcal{I}} } 
\int_{ \mathbb{R}^{n_\mathcal{I}} } 
e^{2\pi  i   [\Phi_\mathcal{I}(x_\mathcal{I}, \xi_\mathcal{I})-y_{\mathcal{I}_1} \cdot \xi_{\mathcal{I}_1} ]}
\prod_{i\in\mathcal{I}_2}
[
e^{-2\pi  i  y_i \cdot \xi_i}-e^{-2\pi  i  \overline{y}_i \cdot \xi_i}
]
G^\nu_{j_\mathcal{I}} a_R (x, y_\mathcal{I}, \xi_\mathcal{I} ) 
 d\xi_\mathcal{I} dy_\mathcal{I} 
 \\\\ \displaystyle
 =
 \int_{ \mathbb{R}^{n_\mathcal{I}} } 
\int_{ \mathbb{R}^{n_\mathcal{I}} } 
e^{2\pi  i   [\Phi_\mathcal{I}(x_\mathcal{I}, \xi_\mathcal{I})-y_{\mathcal{I}_1} \cdot \xi_{\mathcal{I}_1} ]}
\prod_{i\in\mathcal{I}_2}
e^{-2\pi  i y_i'\cdot \xi_i} (-2\pi \i \xi_i)|y_i-\overline{y}_i|

G^\nu_{j_\mathcal{I}} a_R (x, y_\mathcal{I}, \xi_\mathcal{I} ) 
 d\xi_\mathcal{I} dy_\mathcal{I} 
 
  \end{array}
\end{equation}
for $y_i'=ty_i +(1-t)\overline{y}_i$,~$0\le t\le 1$. Once we observe that $|\xi_i|$ bounded by $C 2^{j_i}$ for $i\in \I$, then repeating the proof above we have 
\begin{equation}
\begin{array}{lc}\displaystyle
\int_{ Q^c_\mathcal{I} } 

\Big\{
\int_{ \mathbb{R}^{n_\mathcal{J}} } 
\Big| T^\nu_{j_\mathcal{I}} a_R(x) \Big|^2 d x_\mathcal{J}
\Big\}^{\frac{1}{2}}
d x_\mathcal{I}
\le  \mathfrak{C}  \prod_{i\in \mathcal{I}} 2^{-j_i\frac{n_i-1}{2}} 
\prod_{i\in \mathcal{I}_1}2^{-j_i+k_i}
\prod_{i\in \mathcal{I}_2}2^{j_i-k_i}
\\\\ \displaystyle
  \prod_{i\in \mathcal{J}_1} |I_i|^{\frac{1}{2} -\frac{1}{2n_i}}
 \int_{ \mathbb{R}^{n_\mathcal{J}} } 
 \Big\{
 \int_{ \mathbb{R}^{n_\mathcal{J}} } 
\Big|a_R (x_\mathcal{I}, x_\mathcal{J} ) 
\Big|^2
dx_{\mathcal{J}} \Big\}^{\frac{1}{2}}
dx_\mathcal{I}.
\end{array}
\end{equation}
Finally,  summing the  above inequality  in $\nu$, and take into account that there are essentially at most a constant multiple $\prod_{i\in \mathcal{I}}2^{j_i(n_i-1)/2}$ terms involved, this proves the desired estimates.
\end{proof}

Now we are in the position to prove our  main  Theorem \ref{main theorem}.
\begin{proof}
First we are going to write 
\[
\sum_{j_\mathcal{I} \ge 0}=\prod_{i\in \mathcal{I}}\sum_{j_i\ge 0}=\sum_{\mathcal{I}_1} 
\sum_{j_{\mathcal{I}_1}>k_{ \mathcal{I}_1}} \sum_{j_{\mathcal{I}_2}\le k_{ \mathcal{I}_2}} ,
\]
where the sum  $\sum_{ \mathcal{I}_1 }$  is taken over all subsets of $\mathcal{I}$. Since there are finitely many $\mathcal{I}_1$'s, it suffices to show 
\[
|Q_{\mathcal{J}_1}|^{\frac{1}{2}} \sum_{j_{\mathcal{I}_1}>k_{ \mathcal{I}_1}} \sum_{j_{\mathcal{I}_2}\le k_{ \mathcal{I}_2}}   \int_{Q^c_{\mathcal{I}}} 
\Big \{\int_{\mathbb{R}^{n_{\mathcal{J}_2}}} |T_{j_\mathcal{I}} a_R (x)|^2 dx_\mathcal{J} \Big\}^{\frac{1}{2}} dx_\mathcal{I} \le C.
\]
Recall the definition of $T_{j_\I}$ in (\ref{rectangular partial operator}),  applying the Lemma  \ref{rectangular majorization} we have 
\begin{equation}
\begin{array}{lc}\displaystyle
|Q_{\mathcal{J}_1}|^{\frac{1}{2}}
 \sum_{j_{\mathcal{I}_1}>k_{ \mathcal{I}_1}} \sum_{j_{\mathcal{I}_2}\le k_{ \mathcal{I}_2}}
    \int_{Q^c_{\mathcal{I}}} 
\Big \{\int_{\mathbb{R}^{n_{\mathcal{J}}}} |T_{j_\mathcal{I}} a_R(x)|^2 dx_\mathcal{J} \Big\}^{\frac{1}{2}} dx_\mathcal{I} 
\\\\ \displaystyle
\le   C
 \sum_{j_{\mathcal{I}_1}>k_{ \mathcal{I}_1}} \sum_{j_{\mathcal{I}_2}\le k_{ \mathcal{I}_2}}
 \Big\{
  \prod_{i\in \mathcal{J}_1}  r_i^{\frac{1}{2}}
\prod_{i\in \mathcal{I}_1}2^{-j_i+k_i}
\prod_{i\in \mathcal{I}_2}2^{j_i-k_i}
\\\\ \displaystyle
\times  \prod_{i\in \mathcal{J}_1} |I_i|^{\frac{1}{2} -\frac{1}{2n_i}}
 \int_{ \mathbb{R}^{n_\mathcal{J}} } 
 \Big\{
 \int_{ \mathbb{R}^{n_\mathcal{J}} } 
\Big|a_R (x_\mathcal{I}, x_\mathcal{J} ) 
\Big|^2
dx_{\mathcal{J}} \Big\}^{\frac{1}{2}}
dx_\mathcal{I}
\Big\}
\\\\ \displaystyle
\le  C
  \prod_{i\in \mathcal{J}_1} |I_i|^{\frac{1}{2}}
 \int_{ \mathbb{R}^{n_\mathcal{J}} } 
 \Big\{
 \int_{ \mathbb{R}^{n_\mathcal{J}} } 
\Big|a_R (x_\mathcal{I}, x_\mathcal{J} ) 
\Big|^2
dx_{\mathcal{J}} \Big\}^{\frac{1}{2}}
dx_\mathcal{I}
\\\\ \displaystyle
\le C 
  \prod_{i\in \K} | I_i|^{\frac{1}{2}}\cdot  \Big\|a_R
\Big\|_{L^2}
\le  C
|R|^{\frac{1}{2}}
\Big\|a_R
\Big\|_{L^2}.
\end{array}
\end{equation}
The last inequality holds since $r_i\ge 1$ for $i\in \mathcal{J}_2$. since $a_R(x)$ is a rectangular atom and recall the inequality (\ref{rectangular aim}),
\begin{equation}
\begin{array}{lc}\displaystyle
\int_{\R^n} |Ta_R(x)|dx= \int_{\R^{n_\K}} \int_{ \R^{\J_2} } |T a_R(x)| dx \le \sum_{\I \subset \K}  \int_{Q^c_{\mathcal{I}}} \int_{Q_{\mathcal{J}_1}} \int_{\mathbb{R}^{n_{\mathcal{J}_2}}} |T  a_R(x)| dx
\\\\ \displaystyle
\le 
\sum_{\I \subset \K}
\sum_{j_\mathcal{I} \ge 0} \int_{Q^c_{\mathcal{I}}} \int_{Q_{\mathcal{J}_1}} \int_{\mathbb{R}^{n_{\mathcal{J}_2}}} |T_{j_\mathcal{I}} a_R(x)| dx
\le  C
|R|^{\frac{1}{2}}
\Big\|a_R
\Big\|_{L^2}
\le C,

\end{array}
\end{equation}
which proves our Theorem \ref{main theorem}.

\end{proof}

\section{$L^2$ estimates}
For completeness, we  prove the following theorem concerning the $L^2$ estimates of $T$.
\begin{lemma}\label{L^2 estimate}
Let $T $ be defined as in Theorem \ref{main theorem},   suppose  $\sigma \in S^{\vec m}(  \mathbb{R}^\vn ), \vec m=(-(n_1-1)/2, -(n_2-1)/2, \cdots, (n_d-1)/2)$,  then we have 
\begin{equation}
 \Big \{
 \int_{ \mathbb{R}^{n_\mathcal{J}} } 
\Big|
L^N\Tilde{G}^\nu_\mathcal{I} a_R (x, y_\mathcal{I}, \xi_\mathcal{I} ) 
\Big|^2
dx_{\mathcal{J}} \Big\}^{\frac{1}{2}}
\le C
\prod_{ i\in \mathcal{I}} 2^{-j_i\frac{n_i-1}{2}}
 \prod_{i\in \mathcal{J}_1} |I_i|^{\frac{1}{2} -\frac{1}{2n_i}}
\Big\|  a_R(y_\mathcal{I}, \cdot)\Big \|_{L^2(\mathbb{R}^{n_\mathcal{J}})}, 
\end{equation}
where
\[
\Big\|  a_R(y_\mathcal{I}, \cdot)\Big \|_{L^2(\mathbb{R}^{n_\mathcal{J}})}=
\Big\{
\int_{\mathbb{R}^{n_\mathcal{J}} }
\Big|  a_R(y_\mathcal{I}, x_\mathcal{J} )\Big |^2 dx_\mathcal{J} \Big\}^{\frac{1}{2}}.
\]
\end{lemma}

\begin{proof}

Now we write 
\begin{equation}
\begin{array}{lc}\displaystyle
L^N \Tilde{ G}^\nu_{j_\mathcal{I}} a_R (x, y_\mathcal{I}, \xi_\mathcal{I} ) 
 =  \int_{ \mathbb{R}^{n_\mathcal{J}} }
 e^{2\pi  i  \Phi_\mathcal{J}(x_\mathcal{J}, \xi_\mathcal{J}) }
 \widehat{a_R}(y_\mathcal{I}, \xi_\mathcal{J}  ) L^N\Big\{ e^{2\pi i  \Psi_{\I}(x_\I,\xi_\I) } \sigma(x,\xi) 
\phi_{ j_\mathcal{I} }(\xi_{\mathcal{I}})    \chi^\nu_{j_\mathcal{I}}(\xi_{\mathcal{I}})
\Big\}  d\xi_\mathcal{J} 
\\\\ \displaystyle
=
\int_{ \mathbb{R}^{n_\mathcal{J}} }
 e^{2\pi  i  \Phi_\mathcal{J}(x_\mathcal{J}, \xi_\mathcal{J}) }
 L^{N} \Big\{   e^{2\pi i \Psi_\I(x_\I,\xi_\I) } \sigma(x,\xi) 
\prod_{i\in \mathcal{J}}(1+|\xi_i|^2)^{\frac{n_i-1}{4}}   \phi_{j_\mathcal{I}}(\xi_\mathcal{I})  
 \chi^\nu_{j_\mathcal{I}}(\xi_{\mathcal{I}})
\Big\}   \widehat{Sa_R}(y_\mathcal{I}, \xi_\mathcal{J}  ) 
 d\xi_\mathcal{J} 
\\\\ \displaystyle
\end{array}
\end{equation}
where 
\[
\widehat{S a_R}(y_\mathcal{I}, \xi_\mathcal{J} )=  \widehat{a_R}(y_\mathcal{I}, \xi_\mathcal{J}  ) \prod_{i\in \mathcal{J}}  (1+|\xi_i|^2 )^{-\frac{n_i-1}{4}}.
\]
Then we see $ |\widehat{S a_R}(y_\mathcal{I}, \xi_\mathcal{J} )|\le  | \widehat{a_R}(y_\mathcal{I}, \xi_\mathcal{J}  ) \prod_{i\in \mathcal{J}_1'}  (1+|\xi_i|^2 )^{-\frac{n_i-1}{4}}|$, where $\mathcal{J}_1'= \mathcal{J}_1 \cap \{i\in \mathcal{J} : n_i\ge 2 \}$ and $\mathcal{J}_2'=\mathcal{J} \setminus \mathcal{J}_1'. $
If we use Plancherel's theorem in the $\xi_{\mathcal{J}_2}$ space, we have
\[
\begin{array}{lc} \displaystyle
\Big\{
\int_{\mathbb{R}^{n_\mathcal{J}} }
\Big| S a_R(y_\mathcal{I}, x_\mathcal{J} )\Big |^2 dx_\mathcal{J} \Big\}^{\frac{1}{2}}
= \Big\{
\int_{\mathbb{R}^{n_\mathcal{J}} }
\Big| \widehat{a_R}(y_\mathcal{I}, \xi_\mathcal{J}  ) \prod_{i\in \mathcal{J}}  (1+|\xi_i|^2 )^{-\frac{n_i-1}{4}}  \Big|^2 
d \xi_\mathcal{J} \Big\} ^{\frac{1}{2}}
\\\\ \displaystyle
\le 
\Big\{
\int_{\mathbb{R}^{n_\mathcal{J}} }
\Big| \widehat{a_R}(y_\mathcal{I}, \xi_\mathcal{J}  ) \prod_{i\in \mathcal{J}'_1}  (1+|\xi_i|^2 )^{-\frac{n_i-1}{4}}  \Big|^2 
d \xi_\mathcal{J} \Big\} ^{\frac{1}{2}}
\\\\ \displaystyle
\le 
\Big\{
\int_{\mathbb{R}^{n_{\mathcal{J}'_1} } }
\int_{\mathbb{R}^{n_{\mathcal{J}'_2}} }
\Big| \widehat{a_R}(y_\mathcal{I}, \xi_{\mathcal{J}'_1} ,x_{\mathcal{J}'_2}     ) \prod_{i\in \mathcal{J}'_1}  (1+|\xi_i|^2 )^{-\frac{n_i-1}{4}}  \Big|^2 
d x_{\mathcal{J}'_2} 
d \xi_{\mathcal{J}'_1} \Big\} ^{\frac{1}{2}}
\\\\ \displaystyle
 \le 
 \Big\{
 \int_{\mathbb{R}^{n_{\mathcal{J}'_2}} }
 \int_{\mathbb{R}^{n_{\mathcal{J}'_1} } }
 \Big| a_R(y_\mathcal{I}, y_{\mathcal{J}'_1},  x_{\mathcal{J}'_2}) \mathcal{R} (x_{\mathcal{J}'_1} -y_{\mathcal{J}'_1}) \Big|^2
 d \xi_{\mathcal{J}'_1} 
 d x_{\mathcal{J}'_2} 
 \Big\} ^{\frac{1}{2}}

 \end{array}
\]
where 
\[
\mathcal{R} (x_{\mathcal{J}'_1})=  \int_{\mathbb{R}^{ n_{\mathcal{J}'_1}}  } 
e^{ 2 \pi  i  x_{\mathcal{J}'_1}\cdot \xi_{\mathcal{J}'_1} }  
 \prod_{i\in \mathcal{J}'_1}  (1+|\xi_i|^2 )^{-\frac{n_i-1}{4}}
d\xi_{\mathcal{J}'_1},
\]
We claim that 
\[
\Big|  \mathcal{R} (x_{\mathcal{J}'_1})  \Big|\le 
C
  \prod_{i\in\mathcal{J}'_1} |x_i|^{-n_i(1-(n_i-1)/2n_i)}.
\]
We denote 
\[
\mathcal{R}_{j_i} (x_i)=   \int_{\mathbb{R}^{ n_i} }
e^{ 2 \pi  i  x_i \cdot \xi_i }  
(1+|\xi_i|^2 )^{-\frac{n_i-1}{4}} \phi_{j_i} (\xi_i)
d\xi_i, \qquad i\in \mathcal{J}'_1.
\]
From the definition of $\phi_{j_i}$, we see for $n_i$-dimensional multi-indices $\alpha_i$,
\begin{equation}\label{differential ineq}
\Big| 
\partial_{\xi_i}^{\alpha_i} (1+|\xi_i|^2)^{-\frac{n_i-1}{4}} \phi_{j_i}(\xi_i)  \Big|\le C 2^{-j_i (n_i-1)/2} 2^{-j_i|\alpha_i|}.
\end{equation}
An $M_i$-fold integration by parts with respect to $\xi_i$  gives,

\begin{equation}\label{kernel size}
    \begin{array}{lr}\displaystyle
|\mathcal{R}_{j_i}(x_i) | \le 
\Big|  
  \int_{\mathbb{R}^{ n_i} } 
e^{ 2 \pi  i  x_i \cdot \xi_i  }  
  (1+|\xi_i|^2 )^{-\frac{n_i-1}{4}}
d\xi_i
\Big| 
\\\\ \displaystyle
\leq C_{M_i}
|x_i|^{-M_i} 
\Big |
  \int_{\mathbb{R}^{ n_i  } }
e^{ 2 \pi  i   x_i \cdot \xi_i }  
\partial_{\xi_i}^{M_i} [
  (1+|\xi_i|^2 )^{-\frac{n_i-1}{4}}]
d\xi_i
\Big|
\\\\ \displaystyle
\le C_{M_i}
2^{-j_i (n_i -1)/2 }   \Big(2^{j_i}|x_i| \Big)^{-M_i}2^{j_i n_i}
 \\\\ \displaystyle
\le C_{M_i}   \Big( 2^{j_i}|x_i|\Big)^{-M_i}2^{ j_in_i (1-(n_i-1)/2n_i)},
\end{array}
\end{equation}
We choose 
\begin{equation}\label{choose parameter}
    \begin{array}{cc}\displaystyle
M_i=0 \quad \hbox{if} \quad |x_i|\leq 2^{-j_i}\quad \hbox{or}\quad M_i=n_i  \quad \hbox{if} \quad  |x_i|>2^{-j_i}, \quad i\in \mathcal{J}'_1.
\end{array}
\end{equation}
From (\ref{kernel size}) and (\ref{choose parameter}) and note that  $n_i\ge 2$ for $i\in \mathcal{J}'_1$, we have

\begin{equation}
  \begin{array}{lr}\displaystyle
|\sum_{j_i \ge 0} \mathcal{R}_{j_i}(x_i)|
\le   C_{M} 
\sum_{j_i \ge 0}\Big\{ \Big(2^{j_i}|x_i|  \Big)^{-M_i}2^{j_in_i (1-(n_i-1)/2n_i)}\Big\}
\\\\ \displaystyle
~\leq~C  \Big \{\Big(
{\frac{1}{|x_i|}}\Big)^{n_i\Big(1-{\frac{(n_i-1)}{2n_i}}
\Big)}+\Big({\frac{1}{|x_i|}} 
\Big)^{n_i} \sum_{|x_i|>2^{-j_i}} 2^{ -j_i(n_i-1)/2} \Big\}
\le C 
\Big({\frac{1}{|x_i|}}\Big)^{n_i\Big(1-{\frac{n_i-1}{2n_i}}\Big)},
\end{array}  
\end{equation}
which proves our claim.

Now  applying Hardy-Littlewood-Sobolev inequality on every coordinate subspace and using Minkowski integral inequality, we have
\[
\begin{array}{lc} \displaystyle
\Big\{
\int_{\mathbb{R}^{n_\mathcal{J}} }
\Big| S a_R(y_\mathcal{I}, x_\mathcal{J} )\Big |^2 dx_\mathcal{J} \Big\}^{\frac{1}{2}}
= \Big\{
\int_{\mathbb{R}^{n_\mathcal{J}} }
\Big| \widehat{a_R}(y_\mathcal{I}, \xi_\mathcal{J}  ) \prod_{i\in \mathcal{J}}  (1+|\xi_i|^2 )^{-\frac{n_i-1}{4}}  \Big|^2 
d \xi_\mathcal{J} \Big\} ^{\frac{1}{2}}
\\\\ \displaystyle
 \le C
 \Big\{
 \int_{\mathbb{R}^{n_{\mathcal{J}'_2}} }
 \int_{\mathbb{R}^{n_{\mathcal{J}'_1} } }
 \Big| a_R(y_\mathcal{I}, y_{\mathcal{J}'_1},  x_{\mathcal{J}'_2}) 
\prod_{i\in \mathcal{J}'_1}|x_i-y_i|^{-n_i \Big(1-\frac{n_i-1}{2n_i} \Big)}
 \Big|^2
 d y_{\mathcal{J}'_1} 
 d x_{\mathcal{J}'_2} 
 \Big\} ^{\frac{1}{2}}
 \\\\ \displaystyle
  \le C \int_{\mathbb{R}^{n_{\mathcal{J}_2}} }
   \Big\{
 \int_{\mathbb{R}^{n_{i_{\mathcal{J}_1 }} } }
 \cdots
  \Big\{
 \int_{\mathbb{R}^{n_{i_2} } }
 \Big\{
 \int_{\mathbb{R}^{n_{i_1} } }
 \Big| a_R(y_\mathcal{I},  x_\mathcal{J}) 
 \Big|^{p_{i_1}} d x_{{i_1}}
 \Big\}^{\frac{p_{i_2}}{p_{i_1}}}
 d x_{{i_2}}
 \Big\}^{\frac{p_{i_3}}{p_{i_2}}}
 \cdots
  d x_{  i_{\mathcal{J}_1 } }
 \Big\}^{\frac{2 }{p_{i_{ \mathcal{J}_1}}}}
  d x_{\mathcal{J}_2} 
 \Big\} ^{\frac{1}{2}}

 \\\\ \displaystyle
 \quad \quad \textit{for} \quad  p_{i_k}=\frac{2n_{i_k}}{2n_{i_k}-1}, \quad  \mathcal{J}_1=\{i_1,i_2, \dots, i_{\mathcal{J}_1}\}
 \\\\ \displaystyle
:=C   \Big\| a_R(y_\mathcal{I}, \cdot,  \cdot) 
 \Big \|_{L_{x_{\mathcal{J}_1}}^{p_{\mathcal{J}_1}}  L^2_{x_{\mathcal{J}_2 }}}.
 \end{array}
\]
Now applying H\"older inequality, as a result,
\begin{equation}\label{multiplier estimate}
\Big\{
\int_{\mathbb{R}^{n_\mathcal{J}} }
\Big| S a_R(y_\mathcal{I}, x_\mathcal{J} )\Big |^2 dx_\mathcal{J} \Big\}^{\frac{1}{2}}
\le C 
\prod_{i\in \mathcal{J}_1} |I_i|^{\frac{1}{2} -\frac{1}{2n_i}}
\Big\{
\int_{\mathbb{R}^{n_\mathcal{J}} }
\Big|  a_R(y_\mathcal{I}, x_\mathcal{J} )\Big |^2 dx_\mathcal{J} \Big\}^{\frac{1}{2}}
\end{equation}

Note that
\begin{equation}\label{symbol estimate}
\begin{array}{lc} \displaystyle
\prod_{ i\in \mathcal{I}} 2^{j_i\frac{n_i-1}{2}}
 L^{N} \Big\{  e^{2\pi  i  \Psi_\I(x_\I,\xi_\I) } \sigma(x,\xi) 
\prod_{i\in \mathcal{J}}(1+|\xi_i|^2)^{\frac{n_i-1}{4}}   \phi_{j_\mathcal{I}}(\xi_\mathcal{I})  
 \chi^\nu_{j_\mathcal{I}}(\xi_{\mathcal{I}})\Big\} 
 \\\\ \displaystyle
 \in S^{\vec m}(  \mathbb{R}^{n_{i_1}} \times \mathbb{R}^{n_{i_2}} \times\cdots\times \mathbb{R}^{n_{i_\mathcal{J}}}) , ~\vec m=(0, 0, \cdots, 0) \in \mathbb{R} ^{|\mathcal{J}|}, ~\mathcal{J}=\{i_1, i_2, \dots, i_{\mathcal{J}} \},
 \end{array}
 \end{equation}
hence,  by the $L^2$-boundedness of Fourier integral operators and Plancherel's theorem we have
\begin{equation}
\begin{array}{lc} \displaystyle
 \Big \{
 \int_{ \mathbb{R}^{n_\mathcal{J}} } 
\Big|
L^N\Tilde{G}^\nu_{j_\mathcal{I}} a_R (x, y_\mathcal{I}, \xi_\mathcal{I} ) 
\Big|^2
dx_{\mathcal{J}} \Big\}^{\frac{1}{2}}
\le C 
\prod_{ i\in \mathcal{I}} 2^{-j_i\frac{n_i-1}{2}}
\Big\{
\int_{\mathbb{R}^{n_\mathcal{J}} }
\Big| S a_R(y_\mathcal{I}, x_\mathcal{J} )\Big |^2 dx_\mathcal{J} \Big\}^{\frac{1}{2}}
\\\\ \displaystyle
\le  C 
\prod_{ i\in \mathcal{I}} 2^{-j_i\frac{n_i-1}{2}}
 \prod_{i\in \mathcal{J}_1} |I_i|^{\frac{1}{2} -\frac{1}{2n_i}}
\Big\{
\int_{\mathbb{R}^{n_\mathcal{J}} }
\Big|  a_R(y_\mathcal{I}, x_\mathcal{J} )\Big |^2 dx_\mathcal{J} \Big\}^{\frac{1}{2}}
\\\\ \displaystyle
= C 
\prod_{ i\in \mathcal{I}} 2^{-j_i\frac{n_i-1}{2}}
 \prod_{i\in \mathcal{J}_1} |I_i|^{\frac{1}{2} -\frac{1}{2n_i}}
\Big\|  a_R(y_\mathcal{I}, \cdot)\Big \|_{L^2(\mathbb{R}^{n_\mathcal{J}})}.
\end{array}
\end{equation}

\end{proof}

\section{Further estimates  on rectangular Hardy space}
In this section, we consider the multi-parameter H\"ormander class $S^\vm_{ \rho, \delta} (\R^\vn ) $,  where $\vm=(m_1, m_2, \dots, m_d)$.
Let a symbol $\sigma(x,\xi) \in \mathcal{C}^\infty (\R^n \times \R^n )$ and $n=n_1+n_2+\cdots+n_d, d\ge 2$. Then we say $\sigma \in S^\vm_{ \rho,  \delta} (\R^\vn )$ if it satisfies the estimates
\begin{equation}\label{multi-parameter class}
\left|\p_\xi^\alpha\p_{x}^\beta \sigma(x,\xi)\right|~\leq~C_{\alpha, \beta}  \prod_{i=1}^d \left(1+|\xi_i|\right)^{m_i- \rho|\alpha_i|+\delta|\beta_i| }
\end{equation}
for every multi-indices $\alpha,\beta$.  Here $\xi=(\xi_1, \xi_2, \dots, \xi_d)\in \R^\vn $ and 
 $x=(x_1, x_2, \dots, x_d)\in \R^\vn $. we see that  $S^\vm (\R^\vn )$  is exactly  $S^\vm_{1, 0} (\R^\vn  )$.

Suppose  now each $\nabla_{\xi_i} \nabla_{\xi_i}    \Phi_i$ has constant rank, we prove the following analogue of one-parameter setting.
\begin{theorem}\label{constant rank theorem}
Let $T$ be  defined as Theorem \ref{main theorem} and  for each $i=1,2, \dots,d$, the  phase function $\Phi_i(x_i, \xi_i)$ satisfies that  the rank of Hessian $ \nabla_{\xi_i} \nabla_{\xi_i} \Phi_i(x_i, \xi_i)$ is ~$t_i, ~0\le t_i \le n_i-1 $ for $x_i$ fixed.
Let ~$ m_i=-t_i/2-(n_i-t_i)(1-\rho)/2$ and  $1/2\le \rho \le 1$ for each $i$ and $\sigma \in S^\vm_{\rho, 1-\rho} (\R^\vn )$,
then we have
\begin{equation}
|| Tf ||_{L^1(\R^\vn )} \le C_{\sigma, \Phi} || f ||_{H_{rect}^1(\R^\vn )}.
\end{equation}
\end{theorem}
\begin{remark}
  If $t_i=0$ for $1\le i\le d$, then $T$ is essentially a pseudo-differential operator, suppose  $\sigma(x, \xi) \in S^\vm_{ \rho,  \rho}$, where $m_i=0,~0\le \rho<1$ for each $i=1,2,\cdots,d$.    Journ\'e in the end of the paper  \cite{J85} asked whether   the operator $T$  is bounded from $L^2(\R^n)$ to itself, then Carbery and Seeger \cite{CS88} answered this question affirmatively.  
If in addition $\sigma \in S^{\vm}_{1,  \delta }$, where $m_i=0, 0 \le \delta \le 1$ for $1\le i\le d$,   Carbery and Seeger \cite{CS92} showed that $T$ is bounded from $H^1(\R^\vn)$ to $L^1(\R^n)$, however, for $\rho<1$, this problem seems entirely open. 
\end{remark}

From the above section we see that it suffices to show 
\begin{lemma}\label{constant rank majorization}
Suppose that $\sigma(x,\xi) $ be defined as in Theorem \ref{constant rank theorem} and we use the same notation as the above section.  Then we have
\begin{equation}
\begin{array}{lc}\displaystyle
|Q_{\J_1}|^{\frac{1}{2}}
\int_{Q_\I^c} 
\left\{ 
 \int_{\R^{n_\J}} |T_{j_\I} a_R (x_\I, x_\J) |^2  dx_\J \right\}^{\frac{1}{2}} dx_\I 
  \le  C \prod_{i\in \mathcal{I}_1}2^{-j_i+k_i} \prod_{i\in \mathcal{I}_2}2^{j_i-k_i}.
\end{array}
\end{equation}
Here the region of influence is defined by (\ref{region of influence}).
\end{lemma}

\begin{proof}
We can assume the symbol  $\sigma(x,\xi)$ is supported in a product  narrow cone $\prod_{i=1}^d \Gamma_i$ in $\xi$-space.
Since the rank of Hessian $ \nabla_{\xi_i} \nabla_{\xi_i}   \Phi_i(x_i, \cdot)$ is ~$t_i$, there is an $t_i$-dimensional submanifold ~$S_{t_i}(x_i)$ ~of ~$\mathbb{S}^{n_i-1} \cap \Gamma_i$, varying smoothly with $x_i$, such that $\mathbb{S}^{n-1} \cap \Gamma_i$ can be parameterized by $\overline{\xi}_i=\overline{\xi}_{x_i}(u_i,v_i)$ for $(u_i,v_i)$ in a bounded open set $U_i\times V_i$ near $(0,0)\in \mathbb{R}^{t_i} \times \mathbb{R}^{n_i -t_i-1}$; Furthermore, $\overline{\xi}_{x_i}(u_i,v_i) \in S_{t_i}(x_i)$ if and only if $v_i=0$, and  $ \nabla_{\xi_i} \Phi_i(x_i,\overline{\xi}_{x_i}(u_i,v_i))  =\nabla_{\xi_i} \Phi_i(x_i,\overline{\xi}_{x_i}(u_i, 0)) $.

For each $j_i>0$, we choose $u_{j_i}^{\nu_i}, ~{\nu_i} =1, 2. \dots, N_{t_i}(j_i)$, such that $|u_{j_i}^{\nu_i}-u_{j_i}^{\nu'_i}  | \ge C 2^{-j_i/2}$ for different $\nu_i, \nu'_i$, and such that $U_i$ is covered by balls with center $u_{j_i}^{\nu_i}$ and radius $ C  2^{-j_i/2}$. Note that $N_{t_i}(j_i)$ is bounded by  a constant multiple of  $2^{j_i t_i/2}$. 

Let $\Gamma_{j_i}^{\nu_i}$ denote the cone in the $\xi_i$-space given by 
\[
\Gamma_{j_i}^{\nu_i}=\Big\{\xi_{i} \in \mathbb{R}^{n_i}: \xi_i =s \overline{\xi}_{x_i}(u_i, v_i), \Big|  \frac{u_i}{|u_i|} -u_{j_i}^{\nu_i}\Big|\le 2\cdot 2^{-j_i /2}, s>0  \Big\}
\]
One also define the following partition of unity on $\Gamma_{j_i}^{\nu_i}$, to construct this,  take
\[
\Tilde{\eta}_{j_i}^{\nu_i}(u_i)=\varphi\Big( 2^{j_i/2} \Big(\frac{u_i}{|u_i|} -u_{j_i}^{\nu_i}\Big)\Big)
\]
and set 
\[
\Tilde{\chi}_{j_i}^{\nu_i}(u_i) :=\frac{\Tilde{\eta}_{j_i}^{\nu_i}(u_i)}{ \sum_{1}^{N_{t_i}(j_i)} \Tilde{\eta}_{j_i}^{\nu_i}(u_i)}.
\]
The corresponding homogeneous partition of unity $\chi_{j_i}^{\nu_i}$ is defined such that
\[
\chi_{j_i}^{\nu_i}(s \overline{\xi}_{x_i}(u_i, v_i) ) =\Tilde{\chi}_{j_i}^{\nu_i}(u_i),
\]
for $v_i\in V_i$ and $s>0$. For sufficiently large $C$ we shall now let $R_{j_i}^{\nu_i}$ be the set of all $x_i$ such that
\[
\Big| \Big<\overline{y}_i-\nabla_{\xi_i} \Phi_i(x_i, \overline{\xi}_{x_i}(u_{j_i}^{\nu_i}, 0)) , e_i'\Big> \Big|\le C 2^{-j_i/2}
\] 
for all unit vectors $e_i'$  which are tangential to $S_{t_i}(x_i)$ at $\overline{\xi}_{x_i}(u_{j_i}^\nu, 0  )$, and 
\[
\Big| \Big<\overline{y}_i-\nabla_{\xi_i} \Phi_i(x_i, \overline{\xi}_{x_i}(u_{j_i}^{\nu_i}, 0)) , e_i''\Big> \Big|\le C  2^{-j_i\rho}
\] 
for all unit vectors $e_i''$  which are normal  to $S_{t_i}(x_i)$ at $\overline{\xi}_{x_i}(u_{j_i}^{\nu_i}, 0  )$. Thus we see 
$R_{j_i}^{\nu_i}$ is now basically a rectangle having $n_i-t_i$ sides of length $ C  2^{-j_i \rho}$ and $t_i$ of length $ C  2^{-j_i/2}$. 

Now the \emph{region of influence} is defined as 
\[
Q_i= \bigcup_{j_i\ge k_i} \bigcup_{\nu=1}^{N_{t_i}(j_i)} R_{j_i}^{\nu_i}
\]
Then 
\[
|Q_i|\le \mathfrak{C}  2^{-k_i\rho(n_i-t_i)} \le \mathfrak{C} r_i^{\rho(n_i-t_i)}.
\]
Now  we define
\begin{equation}\label{region of influence}
Q_i^c=\mathbb{R}^{n_i} \setminus Q_i, ~~~~Q_{\mathcal{I}}^c= \bigotimes_{i\in\mathcal{I}} Q_i^c, ~~~~Q_{\mathcal{J}_1}= \bigotimes_{i\in\mathcal{J}_1} Q_i.
\end{equation}
Recall that 
\begin{equation}\label{T E connection}
\begin{array}{lc} \displaystyle
T_{j_\I}^\nu a_R (z_\I, x_\J)= \int_{\R^{n_\I}}\prod_{i\in \I}  \Big( 1+|2^{j_i/2}(x_i-y_i)' |^2 \Big)^{-N} \Big( 1+  |2^{j_i\rho} (x_i -y_i)'' |^2 \Big)^{-N} 
\\\\ \displaystyle
\times  \int_{\R^{n_\I}} e^{2\pi i (x_\I- y_\I) \cdot \xi_\I}  
L^N\Tilde{G}^\nu_{j_\mathcal{I}} a_R (z_\I, x_\J, y_\mathcal{I}, \xi_\mathcal{I} ) 
 d\xi_\mathcal{I} 
 dy_\mathcal{I} .
\end{array}
\end{equation}
where
\[
x_i=\nabla_{\xi_i} \Phi_i(z_i, \overline{ \xi}_{x_i}(u_{j_i}^{\nu_i}, 0) ), \qquad i\in \I.
\]
Denote 
\[
n_i'=t_i,~~n''_i=n_i-t_i,~~n'_{\I}=\sum_{i\in \I} n'_i,~~n''_{\I}=\sum_{i\in \I} n''_i,~~
\]
\[
\xi'_\I=(\xi'_i )_{i\in \I}\in \R^{n'_\I}, ~~\xi''_\I=(\xi''_i )_{i\in \I}\in \R^{n''_\I}, ~~x'_\I=(x'_i )_{i\in \I}\in \R^{n'_\I}, ~~x''_\I=(x''_i )_{i\in \I}\in \R^{n''_\I}.
\]
Now let 
\[
\begin{array}{lc} \displaystyle
F_{j_\I}^\nu a_R(z_\I, x_\J,  y_\I)=
  \int_{\R^{n_\I}} e^{2\pi i  x_\I\cdot \xi_\I}  
L^N\Tilde{G}^\nu_{j_\mathcal{I}} a_R (z_\I, x_\J, y_\mathcal{I}, \xi_\mathcal{I} ) 
 d\xi_\mathcal{I} 
 \\\\ \displaystyle
 =  \prod_{ i\in \mathcal{I}} 2^{j_i m_i} \int_{\R^{n_\I}} e^{2\pi  i x_\I \cdot \xi_\I} 
 \left\{
 \int_{\R^{n_\J}} e^{2\pi \i \phi_\J(x_\J, \xi_\J) } \Hat{a}_R(y_\I, \xi_\J) 
  \prod_{ i\in \mathcal{I}} 2^{-j_i m_i} \Theta_{j_\I}^\nu(z_\I, x_\J,\xi) d\xi_\J \right\}
  d\xi_\I
   \\\\ \displaystyle
 =  \prod_{ i\in \mathcal{I}} 2^{j_i m_i  } \int_{\R^{n'_\I}} e^{2\pi i x'_\I \cdot \xi'_\I} 
E_{j_\I}^\nu a_R( z_\I, x_\J, y_\I, \xi)
  d\xi'_\I,
 \end{array}
\]
where 
\[
\begin{array}{lc}\displaystyle
E_{j_\I}^\nu a_R( z_\I, x_\J, y_\I, \xi'_\I )=  \int_{\R^{n''_\I}}  \int_{\R^{n_\J}}  e^{2\pi i [ x''_\I \cdot \xi''_\I +\Phi_\J(x_\J, \xi_\J) ]} 
\Hat{a}_R(y_\I, \xi_\J) 
\\\\ \displaystyle
\times   \prod_{ i\in \mathcal{I}} 2^{-j_i m_i} \Theta_{j_\I}^\nu(z_\I, x_\J,\xi)   d\xi''_\I d\xi_\J. 
  \end{array}
\]
 Let $g_i$ be a smooth function defined on $\mathbb{R}^{n''_i } $,
such that $\Hat{g}_i(\xi_i'')=1$ for $|\xi_i''|\le 8$ and $\Hat{g}_i(\xi_i'')=0$ for $|\xi_i''|\ge 16$, let $\Hat{g}_{j_i}(\xi_i'')=\Hat{g}_i(2^{-j_i}\xi_i'' ) $ and $\Hat{g}_{j_\I}(\xi''_\I) =\prod_{i\in \I} \Hat{g}_{j_i}(\xi_i'')$.
 Then we have  
\[
\begin{array}{lc} \displaystyle
E_{j_\I}^\nu a_R( z_\I, x_\J, y_\I, \xi'_\I)= 
 \\\\ \displaystyle
 \int_{\R^{n''_\I}}  \int_{\R^{n_\J}}  e^{2\pi i [ x''_\I \cdot \xi''_\I +\Phi_\J(x_\J, \xi_\J) ]} 
\Hat{a}_R(y_\I, \xi_\J)  \Hat{g}_{j_\I}(\xi''_\I)
  \prod_{ i\in \mathcal{I}} 2^{j_i m_i} \Theta_{j_\I}^\nu(z_\I, x_\J,\xi)   d\xi''_\I d\xi_\J. 
 \end{array}
\]
Recall the estimate of $\Theta_{j_\I}^\nu$ in (\ref{symbol estimate})
\[
\begin{array}{lc} \displaystyle
\prod_{ i\in \mathcal{I}} 2^{-j_i m_i}
 L^{N} \Big\{  e^{2\pi i \Psi_\I(x_\I,\xi_\I) } \sigma(x,\xi) 
\prod_{i\in \mathcal{J}}(1+|\xi_i|^2)^{-m_i/2}   \phi_{j_\mathcal{I}}(\xi_\mathcal{I})  
 \chi^\nu_{j_\mathcal{I}}(\xi_{\mathcal{I}})\Big\} 
 \in S_{\rho,  1-\rho}^{\vec m}(  \mathbb{R}^\vn) ,
 \end{array}
\]
where $~\vec m=(0, 0, \cdots, 0) \in \mathbb{R} ^{d},   1/2 \le \rho\le1$. We see  pseudo-differential operators with symbol in the aforementioned  symbol class are  bounded on $L^2$, see  $2.5$, chapter VII  in \cite{SSS91},  Similarly, so do Fourier integral operators. Therefore,

\begin{equation}
\begin{array}{lc}\displaystyle
\left\{
\int_{\R^{n''_\I}} \int_{\R^{n_\J}} |E_{j_\I}^\nu a_R (z_\I, x_\J, y_\I , \xi'_\I) |^2 dx''_\I dx_\J \right\}^{\frac{1}{2}}
\\\\ \displaystyle
\le C ||g_{j_\I}||_{L^2 (\R^{n''_\I} )} \left\{ \int_{\R^{n_\J}} |\Hat{a}_R(y_\I, \xi_\J) \prod_{i\in \mathcal{J}}(1+|\xi_i|^2)^{m_i/2}   |^2 d\xi_\J  \right\}^{\frac{1}{2}} 
\\\\ \displaystyle
\le C ||g_{j_\I}||_{L^2 (\R^{n''_\I} )}
 \prod_{i\in \mathcal{J}_1} |I_i|^{-\frac{m_i}{n_i}}
\Big\{
\int_{\mathbb{R}^{n_\mathcal{J}} }
\Big|  a_R(y_\mathcal{I}, x_\mathcal{J} )\Big |^2 dx_\mathcal{J} \Big\}^{\frac{1}{2}} 
\\ ~~~~~~~\text{by repeating the proof of estimate of (\ref{multiplier estimate})}
\\\\ \displaystyle
\le C  \prod_{i\in \I} 2^{j_i n''_i/2}
 \prod_{i\in \mathcal{J}_1} |I_i|^{-\frac{m_i}{n_i}}
\Big\{
\int_{\mathbb{R}^{n_\mathcal{J}} }
\Big|  a_R(y_\mathcal{I}, x_\mathcal{J} )\Big |^2 dx_\mathcal{J} \Big\}^{\frac{1}{2}} 
\end{array}
\end{equation}
Observe that $E_{j_\I}^\nu a_R( z_\I, x_\J, y_\I, \xi'_\I )=0$ for all $\xi'_\I$ outside a set of measure bounded by $C \prod_{i \in \I} 2^{j_i n'_i/2}$, thus we have 
\begin{equation}\label{F estimate}
\begin{array}{lc}\displaystyle
\left\{
\int_{\R^{n''_\I}} \int_{\R^{n_\J}} |F_{j_\I}^\nu a_R (z_\I, x_\J, y_\I) |^2 dx''_\I dx_\J \right\}^{\frac{1}{2}}
 \\\\ \displaystyle
\le C   \prod_{i \in \I} 2^{j_i(n_i/2+m_i)}
 \prod_{i\in \mathcal{J}_1} |I_i|^{-\frac{m_i}{n_i}}
\Big\{
\int_{\mathbb{R}^{n_\mathcal{J}} }
\Big|  a_R(y_\mathcal{I}, x_\mathcal{J} )\Big |^2 dx_\mathcal{J} \Big\}^{\frac{1}{2}}.
\end{array}
\end{equation}
Form (\ref{T E connection}) and the definition of $Q_\I$ and $F^\nu_{j_\I}$,
we have 
\begin{equation}
\begin{array}{lc}\displaystyle
\int_{Q_\I^c} 
\left\{ 
 \int_{\R^{n_\J}} |T_{j_\I}^\nu a_R (z_\I, x_\J) |^2  dx_\J \right\}^{\frac{1}{2}} dx_\I 
 \\\\ \displaystyle
 \le  C \prod_{i\in \mathcal{I}_1}2^{-j_i+k_i} \int_{\R^{n_\I}} 
  \int_{\R^{n'_\I}} \prod_{i\in \I}  \Big( 1+|2^{j_i/2}(x_i-y_i)' |^2 \Big)^{-N+1} 
\\\\ \displaystyle
\times
\left\{ \int_{\R^{n''_\I}} 
 \Big( 1+  |2^{j_i\rho} (x_i -y_i)'' |^2 \Big)^{-N+1} 
\left\{\int_{\R^{n_\J}}|F^\nu_{j_\I} a_R (z_\I, x_\J, y_\I) |^2 dx_\J \right\} ^{\frac{1}{2}}dx''_\I  \right\} dx'_\I  dy_\I
\end{array}
\end{equation}
Applying Cauchy-Schwartz inequality and (\ref{F estimate}),

\begin{equation}
\begin{array}{lc}\displaystyle
\int_{Q_\I^c} 
\left\{ 
 \int_{\R^{n_\J}} |T_{j_\I}^\nu a_R (z_\I, x_\J) |^2  dx_\J \right\}^{\frac{1}{2}} dx_\I 
 \\\\ \displaystyle
  \le  C \prod_{i\in \mathcal{I}_1}2^{-j_i+k_i} \prod_{i \in \I} 2^{-j_i\rho(n_i-t_i)/2}
 \int_{\R^{n_\I}} 
  \int_{\R^{n'_\I}} \prod_{i\in \I}   \Big( 1+|2^{j_i/2}(x_i-y_i)' |^2 \Big)^{-N+1} 
\\\\ \displaystyle
\times
\left\{ \int_{\R^{n''_\I}} 
\int_{\R^{n_\J}}|F^\nu_{j_\I} a_R (z_\I, x_\J, y_\I) |^2 dx_\J dx''_\I  \right\} ^{\frac{1}{2}}dx'_\I dy_\I
\\\\ \displaystyle
  \le  C \prod_{i\in \mathcal{I}_1}2^{-j_i+k_i}\prod_{i \in \I}  2^{-j_i n'_i/2}    
  \prod_{i\in \mathcal{J}_1} |I_i|^{-\frac{m_i}{n_i}}
  \int_{\R^{n_\I}}
\Big\{
\int_{\mathbb{R}^{n_\mathcal{J}} }
\Big|  a_R(y_\mathcal{I}, x_\mathcal{J} )\Big |^2 dx_\mathcal{J} \Big\}^{\frac{1}{2}} dy_\I.
\end{array}
\end{equation}
Take into account that there are essentially at most a constant multiple $\prod_{i\in \I}2^{n'_i t_i/2}$ terms involved, which implies 
\begin{equation}
\begin{array}{lc}\displaystyle
|Q_{\J_1}|^{\frac{1}{2}}
\int_{Q_\I^c} 
\left\{ 
 \int_{\R^{n_\J}} |T_{j_\I} a_R (z_\I, x_\J) |^2  dx_\J \right\}^{\frac{1}{2}} dx_\I 
 \\\\ \displaystyle
  \le  C \prod_{i\in \mathcal{I}_1}2^{-j_i+k_i} 
  \prod_{i\in \mathcal{J}_1} |I_i|^{\frac{1}{2}}
  \int_{\R^{n_\I}}
\Big\{
\int_{\mathbb{R}^{n_\mathcal{J}} }
\Big|  a_R(y_\mathcal{I}, x_\mathcal{J} )\Big |^2 dx_\mathcal{J} \Big\}^{\frac{1}{2}} dy_\I
 \\\\ \displaystyle
  \le  C \prod_{i\in \mathcal{I}_1}2^{-j_i+k_i} 
|R|^{\frac{1}{2 }} || a_R||_{L^2(\R^n)}
  \le  C \prod_{i\in \mathcal{I}_1}2^{-j_i+k_i} .
\end{array}
\end{equation}
The inequality holds since $r_i>1$ for $i\in \J_2$ and by applying Cauchy-Schwartz inequality. Again, since $\int_{\R^{n_i}} a_R(x) dx_i=0$ for each $i =1,2,\dots,d$, repeating the argument in  lemma \ref{rectangular majorization}, we have 
\begin{equation}
\begin{array}{lc}\displaystyle
\int_{Q_\I^c} 
\left\{ 
 \int_{\R^{n_\J}} |T_{j_\I} a_R (z_\I, x_\J) |^2  dx_\J \right\}^{\frac{1}{2}} dx_\I 
  \le  C \prod_{i\in \mathcal{I}_1}2^{-j_i+k_i} \prod_{i\in \mathcal{I}_2}2^{j_i-k_i}.
\end{array}
\end{equation}
Note that $dx_\I \le \mathfrak{C} dz_\I$ since by our assumption that each $\Phi_i(x_i,\xi_i)$ satisfies the non-degeneracy condition. Therefore
\begin{equation}
\begin{array}{lc}\displaystyle
\int_{Q_\I^c} 
\left\{ 
 \int_{\R^{n_\J}} |T_{j_\I} a_R (x_\I, x_\J) |^2  dx_\J \right\}^{\frac{1}{2}} dx_\I 
  \le  C \prod_{i\in \mathcal{I}_1}2^{-j_i+k_i} \prod_{i\in \mathcal{I}_2}2^{j_i-k_i}.
\end{array}
\end{equation}

\end{proof}

\addcontentsline{toc}{section}{Acknowledgements}		
%\section*{Acknowledgements}	
	
	\bibliography{references}

\begin{thebibliography}{MaTTT06}

\bibitem[BLRW20]{BLRW20}
O.~Bajas, E.~Latorre, M.~Rinc\'on, and J.~Wright.
\newblock A class of multiparameter oscillatory singular integral operators:
  endpoint {H}ardy space bounds.
\newblock {\em Rev. Mat. Iberoam.}, 36(2):611--639, 2020.

\bibitem[Car74]{C74}
L.~Carleson.
\newblock {\em A Counter Example for Measures Bounded on H\_p for the Bi-disc}.
\newblock Report No. 7 - 1974. Institut Mittag-Leffler, 1974.

\bibitem[CF80]{CF80}
S.~Y.~A. Chang and R.~Fefferman.
\newblock A continuous version of duality of {$H\sp{1}$} with {BMO} on the
  bidisc.
\newblock {\em Ann. of Math. (2)}, 112(1):179--201, 1980.

\bibitem[CF82]{CF82}
S.~Y.~A. Chang and R.~Fefferman.
\newblock The {C}alder\'{o}n-{Z}ygmund decomposition on product domains.
\newblock {\em Amer. J. Math.}, 104(3):455--468, 1982.

\bibitem[CF85]{CF85}
S.~Y.~A. Chang and R.~Fefferman.
\newblock Some recent developments in {F}ourier analysis and {$H^p$}-theory on
  product domains.
\newblock {\em Bull. Amer. Math. Soc. (N.S.)}, 12(1):1--43, 1985.

\bibitem[CS88]{CS88}
A.~Carbery and A.~Seeger.
\newblock Conditionally convergent series of linear operators on {$L^p$}-spaces
  and {$L^p$}-estimates for pseudodifferential operators.
\newblock {\em Proc. London Math. Soc. (3)}, 57(3):481--510, 1988.

\bibitem[CS92]{CS92}
A.~Carbery and A.~Seeger.
\newblock {$H^p$}- and {$L^p$}-variants of multiparameter
  {C}alder\'{o}n-{Z}ygmund theory.
\newblock {\em Trans. Amer. Math. Soc.}, 334(2):719--747, 1992.

\bibitem[\'E70]{E70}
G.~I. \'Eskin.
\newblock Degenerate elliptic pseudodifferential equations of principal type.
\newblock {\em Mat. Sb. (N.S.)}, 82(124):585--628, 1970.

\bibitem[Fef71]{F71}
C.~Fefferman.
\newblock Characterizations of bounded mean oscillation.
\newblock {\em Bull. Amer. Math. Soc.}, 77:587--588, 1971.

\bibitem[Fef85a]{Fe85}
R.~Fefferman.
\newblock The atomic decomposition of {$H^1$} in product spaces.
\newblock {\em Adv. in Math.}, 55(1):90--100, 1985.

\bibitem[Fef85b]{F85}
R.~Fefferman.
\newblock Singular integrals on product {$H^p$} spaces.
\newblock {\em Rev. Mat. Iberoamericana}, 1(2):25--31, 1985.

\bibitem[Fef86]{F86}
R.~Fefferman.
\newblock Calder\'{o}n-{Z}ygmund theory for product domains: {$H^p$} spaces.
\newblock {\em Proc. Nat. Acad. Sci. U.S.A.}, 83(4):840--843, 1986.

\bibitem[Fef87]{F87}
R.~Fefferman.
\newblock Harmonic analysis on product spaces.
\newblock {\em Ann. of Math. (2)}, 126(1):109--130, 1987.

\bibitem[FS72]{FS72}
C.~Fefferman and E.~M. Stein.
\newblock {$H\sp{p}$} spaces of several variables.
\newblock {\em Acta Math.}, 129(3-4):137--193, 1972.

\bibitem[GS79]{GS79}
R.~F. Gundy and E.~M. Stein.
\newblock {$H\sp{p}$} theory for the poly-disc.
\newblock {\em Proc. Nat. Acad. Sci. U.S.A.}, 76(3):1026--1029, 1979.

\bibitem[H\"71]{H71}
L.~H\"ormander.
\newblock Fourier integral operators. {I}.
\newblock {\em Acta Math.}, 127(1-2):79--183, 1971.

\bibitem[Jou85]{J85}
J.~Journ\'e.
\newblock Calder\'{o}n-{Z}ygmund operators on product spaces.
\newblock {\em Rev. Mat. Iberoamericana}, 1(3):55--91, 1985.

\bibitem[Jou86]{J86}
J.~Journ\'e.
\newblock A covering lemma for product spaces.
\newblock {\em Proc. Amer. Math. Soc.}, 96(4):593--598, 1986.

\bibitem[MaTTT04]{MPTT04}
C.~Muscalu, J.~Pipher abd T.~Tao, and C.~Thiele.
\newblock Bi-parameter paraproducts.
\newblock {\em Acta Math.}, 193(2):269--296, 2004.

\bibitem[MaTTT06]{MPTT06}
C.~Muscalu, J.~Pipher abd T.~Tao, and C.~Thiele.
\newblock Multi-parameter paraproducts.
\newblock {\em Rev. Mat. Iberoam.}, 22(3):963--976, 2006.

\bibitem[MRS95]{MRS95}
D.~M\"uller, F.~Ricci, and E.~M. Stein.
\newblock Marcinkiewicz multipliers and multi-parameter structure on
  {H}eisenberg (-type) groups. {I}.
\newblock {\em Invent. Math.}, 119(2):199--233, 1995.

\bibitem[MRS96]{MRS96}
D.~M\"uller, F.~Ricci, and E.~M. Stein.
\newblock Marcinkiewicz multipliers and multi-parameter structure on heisenberg
  (-type) groups. ii.
\newblock {\em Math. Z.}, 221(2):267--291, 1996.

\bibitem[SSS91]{SSS91}
A.~Seeger, C.~Sogge, and E.~M. Stein.
\newblock Regularity properties of {F}ourier integral operators.
\newblock {\em Ann. of Math. (2)}, 134(2):231--251, 1991.

\bibitem[Ste93]{S93}
E.~M. Stein.
\newblock {\em Harmonic analysis: real-variable methods, orthogonality, and
  oscillatory integrals}, volume~43 of {\em Princeton Mathematical Series}.
\newblock Princeton University Press, Princeton, NJ, 1993.

\bibitem[Wan22]{W22}
Z.~Wang.
\newblock Regularity of multi-parameter fourier integral operator, 2022.

\bibitem[Wan23]{W23}
Z.~Wang.
\newblock Singular integrals of non-convolution type on product spaces, 2023.

\end{thebibliography}


\begin{thebibliography}{MaTTT06}

\bibitem[Bea82]{B82}
R.~M. Beals.
\newblock {$L\sp{p}$} boundedness of {F}ourier integral operators.
\newblock {\em Mem. Amer. Math. Soc.}, 38(264), 1982.

\bibitem[Bre77]{B77}
P.~Brenner.
\newblock {$L\sb{p}-L\sb{p'}$}-estimates for {F}ourier integral operators
  related to hyperbolic equations.
\newblock {\em Math. Z.}, 152(3):273--286, 1977.

\bibitem[DH72]{DH72}
J.~J. Duistermaat and L.~H\"ormander.
\newblock Fourier integral operators. {II}.
\newblock {\em Acta Math.}, 128(3-4):183--269, 1972.

\bibitem[dVF76]{CF76}
Y.~Colin de~Verdi\`ere and M.~Frisch.
\newblock R\'{e}gularit\'{e} lipschitzienne et solutions de l'\'{e}quation des
  ondes sur une vari\'{e}t\'{e} riemannienne compacte.
\newblock {\em Ann. Sci. \'{E}cole Norm. Sup. (4)}, 9(4):539--565, 1976.

\bibitem[\'E70]{E70}
G.~I. \'Eskin.
\newblock Degenerate elliptic pseudodifferential equations of principal type.
\newblock {\em Mat. Sb. (N.S.)}, 82(124):585--628, 1970.

\bibitem[Fef71]{F71}
C.~Fefferman.
\newblock Characterizations of bounded mean oscillation.
\newblock {\em Bull. Amer. Math. Soc.}, 77:587--588, 1971.

\bibitem[Fef85]{F85}
R.~Fefferman.
\newblock Singular integrals on product {$H^p$} spaces.
\newblock {\em Rev. Mat. Iberoamericana}, 1(2):25--31, 1985.

\bibitem[Fef86]{F86}
R.~Fefferman.
\newblock Calder\'{o}n-{Z}ygmund theory for product domains: {$H^p$} spaces.
\newblock {\em Proc. Nat. Acad. Sci. U.S.A.}, 83(4):840--843, 1986.

\bibitem[Fef87]{F87}
R.~Fefferman.
\newblock Harmonic analysis on product spaces.
\newblock {\em Ann. of Math. (2)}, 126(1):109--130, 1987.

\bibitem[FP97]{FP97}
R.~Fefferman and J.~Pipher.
\newblock Amer. j. math.
\newblock {\em American Journal of Mathematics}, 119(2):337--369, 1997.

\bibitem[FP05]{FP05}
R.~Fefferman and J.~Pipher.
\newblock A covering lemma for rectangles in {${\Bbb R}^n$}.
\newblock {\em Proc. Amer. Math. Soc.}, 133(11):3235--3241, 2005.

\bibitem[FS72]{FS72}
C.~Fefferman and E.~M. Stein.
\newblock {$H\sp{p}$} spaces of several variables.
\newblock {\em Acta Math.}, 129(3-4):137--193, 1972.

\bibitem[H\"71]{H71}
L.~H\"ormander.
\newblock Fourier integral operators. {I}.
\newblock {\em Acta Math.}, 127(1-2):79--183, 1971.

\bibitem[MaTTT04]{MPTT04}
C.~Muscalu, J.~Pipher abd T.~Tao, and C.~Thiele.
\newblock Bi-parameter paraproducts.
\newblock {\em Acta Math.}, 193(2):269--296, 2004.

\bibitem[MaTTT06]{MPTT06}
C.~Muscalu, J.~Pipher abd T.~Tao, and C.~Thiele.
\newblock Multi-parameter paraproducts.
\newblock {\em Rev. Mat. Iberoam.}, 22(3):963--976, 2006.

\bibitem[Miy80]{M80}
A.~Miyachi.
\newblock On some estimates for the wave equation in {$L\sp{p}$} and
  {$H\sp{p}$}.
\newblock {\em J. Fac. Sci. Univ. Tokyo Sect. IA Math.}, 27(2):331--354, 1980.

\bibitem[MRS95]{MRS95}
D.~M\"uller, F.~Ricci, and E.~M. Stein.
\newblock Marcinkiewicz multipliers and multi-parameter structure on
  {H}eisenberg (-type) groups. {I}.
\newblock {\em Invent. Math.}, 119(2):199--233, 1995.

\bibitem[MRS96]{MRS96}
D.~M\"uller, F.~Ricci, and E.~M. Stein.
\newblock Marcinkiewicz multipliers and multi-parameter structure on heisenberg
  (-type) groups. ii.
\newblock {\em Math. Z.}, 221(2):267--291, 1996.

\bibitem[Per80]{P80}
J.~Peral.
\newblock {$L\sp{p}$} estimates for the wave equation.
\newblock {\em J. Functional Analysis}, 36(1):114--145, 1980.

\bibitem[Sog93]{Sogge93}
C.~Sogge.
\newblock {\em Fourier integrals in classical analysis}, volume 105 of {\em
  Cambridge Tracts in Mathematics}.
\newblock Cambridge University Press, Cambridge, 1993.

\bibitem[SSS91]{SSS91}
A.~Seeger, C.~Sogge, and E.~M. Stein.
\newblock Regularity properties of {F}ourier integral operators.
\newblock {\em Ann. of Math. (2)}, 134(2):231--251, 1991.

\bibitem[Ste93]{S93}
E.~M. Stein.
\newblock {\em Harmonic analysis: real-variable methods, orthogonality, and
  oscillatory integrals}, volume~43 of {\em Princeton Mathematical Series}.
\newblock Princeton University Press, Princeton, NJ, 1993.

\bibitem[Tao04]{T04}
T.~Tao.
\newblock The weak-type {$(1,1)$} of {F}ourier integral operators of order
  {$-(n-1)/2$}.
\newblock {\em J. Aust. Math. Soc.}, 76(1):1--21, 2004.

\bibitem[Wan22]{W22}
Z.~Wang.
\newblock Regularity of multi-parameter fourier integral operator, 2022.

\end{thebibliography}
	\bibliographystyle{alpha.bst}
\end{document}